\documentclass[a4paper,11pt]{article}    
\usepackage{amsmath, amssymb, amsthm,pdfsync}
\usepackage{fullpage}
\usepackage{verbatim}
\usepackage{graphicx}
\usepackage{subfig, caption}
\usepackage{cancel}
\usepackage{enumitem}
\usepackage{wrapfig}
\usepackage{mathtools}
\usepackage{xcolor}
\usepackage{epigraph}
\usepackage{cleveref}
\usepackage{overpic}
\usepackage{esint}
\usepackage{dsfont}
\usepackage{color}
\usepackage{comment}
\usepackage[normalem]{ulem}

\newcommand{\farc}{\frac}

\setlist[itemize]{itemsep=-1mm}

\newtheorem{theorem}{Theorem}
\newtheorem{lemma}[theorem]{Lemma}

\newtheorem{prop}[theorem]{Proposition}
\newtheorem{corollary}[theorem]{Corollary}

\newcommand{\vertiii}[1]{{\Big\vert\kern-0.25ex\Big\vert\kern-0.25ex\Big\vert #1 
    \Big\vert\kern-0.25ex\Big\vert\kern-0.25ex\Big\vert}}

\newcommand{\R}{\mathbb{R}}

\newcommand{\cP}{\mathcal{P}}

\newcommand{\eps}{\varepsilon}
\newcommand{\e}{\eps}

\newcommand{\1}{\mathds{1}}

\newcommand{\Rm}{{\mathbb R}}

\numberwithin{equation}{section}
\numberwithin{theorem}{section}

\crefname{assumption}{assumption}{assumptions}
\crefname{theorem}{theorem}{theorems}
\crefname{lem}{lemma}{lemmas}
\crefname{cor}{corollary}{corollaries}
\crefname{prop}{proposition}{propositions}
\Crefname{theorem}{Theorem}{Theorems}
\crefname{conjecture}{conjecture}{conjectures}
\begin{document}
\title{The Bramson delay in the non-local Fisher-KPP equation}

%

\author{Emeric Bouin 
\footnote{CEREMADE - Universit\'e Paris-Dauphine, UMR CNRS 7534, Place du Mar\'echal de Lattre de Tassigny, 75775 Paris Cedex 16, France. E-mail: \texttt{bouin@ceremade.dauphine.fr}}\and  
Christopher Henderson \footnote{Department of Mathematics, University of Arizona, Tucson, AZ 85721. E-mail: \texttt{ckhenderson@math.arizona.edu}}\and
Lenya Ryzhik \footnote{Department of Mathematics, Stanford University, Stanford, CA 94305, E-mail: \texttt{ryzhik@math.stanford.edu}}}

\maketitle

\maketitle
\begin{abstract}
We consider the non-local Fisher-KPP equation modeling a population with individuals competing with each other for 
resources with a strength related to their distance, and obtain the asymptotics for the position of the 
invasion front starting from a localized population.  Depending on the behavior of the competition kernel at infinity, the location of the front is 
either $2t - ({3}/2)\log t + O(1)$, as in the local case, or $2t - O(t^\beta)$ for some explicit $\beta \in (0,1)$.  Our main tools here are a
local-in-time Harnack inequality and an analysis of the linearized problem with a suitable moving Dirichlet boundary condition.  Our analysis also yields,  for any $\beta\in(0,1)$, examples of Fisher-KPP type non-linearities $f_\beta$ such that the front for the   local
Fisher-KPP equation with reaction term $f_\beta$ is at $2t - O(t^\beta)$.
\end{abstract}

\begin{abstract}
Dans cet article, nous consid\'erons l'\'equation de Fisher-KPP non locale, qui mod\'elise la dynamique d'une population ou la force de comp\'etition pour les ressources d\'epend de la distance entre les individus. Nous obtenons une asymptotique pr\'ecise en temps long de la position d'une population qui est initialement localis\'ee en espace. Selon la d\'ecroissance \`a l'infini du noyau de comp\'etition, la position du front est soit $2t - ({3}/2)\log t + O(1)$, comme dans le cas de l'\'equation locale, soit $2t - O(t^\beta)$, pour un $\beta \in (0,1)$ calcul\'e explicitement. Les outils les plus importants utilis\'es dans cet article sont une version locale en temps d'une in\'egalit\'e de Harnack parabolique ainsi qu'une analyse fine du probl\`eme lin\'earis\'e avec une condition de bord de Dirichlet dynamique. Notre analyse donne aussi, pour tout $\beta\in(0,1)$, des exemples de non-lin\'earit\'es de type Fisher-KPP pour lesquelles le front se trouve en $2t - O(t^\beta)$.
\end{abstract}
\noindent{\bf AMS Class. No:} {35K57, 35Q92, 45K05, 35C07}


\section{Introduction}

The Fisher-KPP equation 
\begin{equation}\label{sep2704}
u_t = u_{xx} + u(1-u)
\end{equation}
is one of the simplest models for population spreading, accounting for
a competition for resources. However, (\ref{sep2704}) only accounts for a
local competition between individuals. When this competition is non-local,
one is led to the non-local Fisher-KPP equation 
\begin{equation}\label{eq:nonlocal}
\begin{aligned} 
	&u_t - u_{xx} = u (1 - \phi \star u), ~~~~~&t>0,~x\in\R,\\
	&u(0,\cdot) = u_0.
\end{aligned}
\end{equation}
Here, $\phi$ is a probability density that represents the strength of the competition between individuals a given distance apart. Equation~\eqref{eq:nonlocal} has garnered much interest recently, mostly for two reasons. First, 
it does not admit a comparison principle, leading to inherent technical difficulties -- even proving a uniform upper bound on $u$ is 
non-trivial~\cite{HamelRyzhik}. Second, unusual behavior may occur, such as the existence of oscillating wave trains behind the front \cite{FayeHolzer,GenieysVolpertAuger,Gourley,NadinRossiRyzhikPerthame}.

Our interest is in the spreading of the solutions of (\ref{eq:nonlocal}) 
when the initial density $u_0$ is localized. 
To motivate our work, we recall the known results for the local 
Fisher-KPP equation~(\ref{sep2704}). 
Going back to the work of Bramson, it is known that if $u_0$  is compactly supported, 
the front of~$u$ is located at 
\begin{equation}\label{sep2702}
	X(t) =2t - \frac32\log t + s_0,
\end{equation}
where $s_0$ is a shift depending only on $u_0$~\cite{Bramson78, Bramson83}, with less precise asymptotics obtained earlier by
Uchiyama~\cite{Uchiyama}.  These proofs have been simplified in recent years~\cite{HNRR_Short,Roberts}, with some
refinements in~\cite{NRR1, NRR2}, and also extended to the 
spatially periodic case~\cite{HNRR_Periodic}. One may think of $\bar X(t)=2t$ as 
the position of a traveling wave, and $d(t)=(3/2)\log t$ as the delay due to the
fact that the initial condition $u_0$ is compactly supported, so that the
solution lags behind the traveling wave.  

In  the non-local case considered
in the present paper, we show that the front position depends on the rate of decay
of the kernel~$\phi$ at infinity. 
When $\phi$ decays fast enough, solutions of (\ref{eq:nonlocal}) spread as those
of the local equation: the front is 
at a position as in (\ref{sep2702}), up to a constant order error. However,  
when $\phi$ decays slowly, and the competition at large distances is relatively 
strong, the delay behind the traveling wave position $2t$ 
is not logarithmic but algebraic, of the order $O(t^\beta)$,  
with~$\beta$ depending only on the rate of decay of~$\phi$.


We now make our assumptions more precise.  First, we assume that $\phi$ 
is an even, continuous, and bounded probability density:
\begin{equation}\label{eq:phi1}
\int_\R \phi(x)dx = 1,~~\hbox{ and \quad $\phi(x) = \phi(-x) \text{ for all } x \in \R$}.
\end{equation}
In addition, $\phi$ has some ``mass'' near the origin, that is, there exists $\sigma_\phi > 0$ such that
\begin{equation}\label{eq:phi2}
	\phi(x) \geq \sigma_\phi \1_{[-\sigma_\phi,\sigma_\phi]}.
\end{equation}
The behavior of $u$ depends strongly on the tail behavior of $\phi$.  Here we make two different  assumptions, that are helpful for 
the upper and lower bounds, respectively.  The first assumption, an upper bound on the tail of $\phi$,
is that there exists $A_\phi > 0$ and $r >1$ such that, for all $R \geq 1$,
\begin{equation}\label{eq:phi_upper}
	\int_R^\infty \phi(x) dx \leq A_\phi R^{-r + 1}.
\end{equation}
Sometimes we will need to complement this with a lower bound on the tail: for all $R\geq 1$, we have
\begin{equation}\label{eq:phi_lower}
	\int_R^\infty \phi(x) dx \geq A_\phi^{-1} R^{-r+1}.
\end{equation}
Roughly,~\eqref{eq:phi_upper} and \eqref{eq:phi_lower} mean that $\phi \sim x^{-r}$ for $x \gg 1$.
%

For the initial condition, we assume that $u_0$ is localized to the left of some point $x_0$:
\begin{equation}\label{eq:u_0}
    0 \leq u_0 \leq 1,\qquad 
    \exists x_0 \text{ such that } u_0(x) = 0 \text{ for all } x \geq x_0,
    \quad \text{ and } \quad
    \liminf_{x\to-\infty} u_0(x) > 0.
\end{equation}
We expect our results to hold when $u_0$ has ``fast" exponential decay, that is, $u_0(x) e^{(1+\epsilon)x} \to 0$ as~$x\to 0$ for some $\epsilon>0$, rather than compactly
supported on the right. However, we recall that the front position asymptotics for 
solutions of (\ref{sep2704}) with $u_0$ that has a sufficiently slow exponential 
tail on the right is 
different from (\ref{sep2702}), see~\cite{Bramson78, Bramson83}.

The main result of this paper is the following.
\begin{theorem}\label{thm:main_delay}
Suppose that $u$ satisfies~\eqref{eq:nonlocal} and~\eqref{eq:u_0} with $\phi$ satisfying~\eqref{eq:phi1},~\eqref{eq:phi2}, and~\eqref{eq:phi_upper}. 
If $r>3$, then the solution $u$ propagates with a logarithmic delay:
        \begin{equation}\label{eq:log_delay_below}
        \liminf_{t\to\infty} \inf_{x \leq 0} u\Big(t,2t - \frac{3}2 \log t +x\Big) > 0,
    \end{equation}
    and
\begin{equation}\label{eq:log_delay_above}
\lim_{L\to\infty} \limsup_{t\to\infty} \sup_{x\geq L}\, u\Big(t,2t - \frac{3}{2}\log t + x\Big) = 0.
\end{equation}
If $ r= 3$, then the solution $u$ propagates with a larger logarithmic delay:  there exists $S_\phi > s_\phi > 3/2$ such that
  \begin{equation}\label{eq:big_log_delay_below}
        \liminf_{t\to\infty} \inf_{x \leq 0} u\Big(t,2t - S_\phi \log t +x\Big) > 0,
    \end{equation}
    and
\begin{equation}\label{eq:big_log_delay_above}
\lim_{t\to\infty} \sup_{ x \geq 0}\, u\Big(t,2t - s_\phi \log t + x\Big) = 0.
\end{equation}
%
%
If $r \in (1,3)$, then the delay is algebraic:  there exists $C_\phi > 0$, depending only on $r$, $\sigma_\phi$, and $A_\phi$, such that
\begin{equation}\label{eq:alg_delay_below}
        \liminf_{t\to\infty}\inf_{x \leq 0} u\Big(t,2t - C_\phi t^\frac{3-r}{1+r} + x\Big) > 0,
    \end{equation}
    and, if additionally~\eqref{eq:phi_lower} holds, then there exists $c_\phi \in (0, C_\phi)$, depending only on $\sigma_\phi$, $r$, and $A_\phi$, such that
	\begin{equation}\label{eq:alg_delay_above}
        \lim_{t\to\infty} \sup_{x \geq 0} u\Big(t,2t - c_\phi t^\frac{3-r}{1+r} + x\Big) = 0.
    \end{equation}
\end{theorem}
As we discuss later in greater detail, heuristically, the competition term $\phi \star u$ acts on the scale~$t^\gamma$,
with~$\gamma = {2}/{(1+r)}$. Note that 
\begin{equation}\label{sep2724}
\frac{3-r}{1+r}=2\gamma -1,
\end{equation}
and that, when $r>3$,~$\gamma < 1/2$, which, in turn, suggests that the competition 
scale is smaller than the diffusive scale~$\sqrt t$.  
This is one way to see that there is a phase transition at $r=3$.   



As a by-product of our analysis, we also obtain results for the local Fisher-KPP equation 
\begin{equation}\label{sep2706}
u_t = u_{xx} + f(u).
\end{equation}
Let us assume that $f$ is of the KPP class: $f(u)/u$ is decreasing in $u$ near $0$, $f\in C^1$, and $f'(0)=1$. A natural question is whether
these assumptions are sufficient to ensure that the front location is given by the logarithmic Bramson correction in (\ref{sep2702}). 
We show, roughly, the following: if 
\[
1-\farc{f(u)}{ u} \sim \left(\log\left(\frac{1}{u}\right)\right)^{1-r} \hbox{ with $r>1$},
\]
then the conclusion of \Cref{thm:main_delay} holds, with  the logarithmic delay for $r\geq3$ and an algebraic delay of the 
order $O(t^{(3-r)/(1+r)})$  for $1<r<3$.
These non-linearities 
are not purely mathematical curiosities: they are regularly used in biology and are known as 
Gompertz models, see~\cite{dennis2006estimating} and the vast body of literature around it.   
The statement and proof of this result are contained in \Cref{sec:local_equation}.

Let us mention a few related works.  The model~\eqref{eq:nonlocal} considered here  was first introduced by 
Britton~\cite{Britton} and has a quite involved history, see the introduction of~\cite{BNPR} for a brief overview.  
The non-local term $\phi \star u$ has different effects depending on whether one is studying the behavior of~$u$ behind the front or at the front.  
Behind the front, there is a possible Turing instability of the steady state of the local Fisher-KPP equation $u \equiv 1$, which complicates the 
behavior.  For example, wave trains have been constructed 
by Faye and Holzer~\cite{FayeHolzer} and, in a related setting, 
in~\cite{NadinRossiRyzhikPerthame}.  Such wave trains have also been observed numerically by Genieys, 
Volpert, and Auger in~\cite{GenieysVolpertAuger}.  As a result, without finer assumptions on $\phi$, one cannot hope for a stronger result than 
the lower bounds in \Cref{thm:main_delay}.  As far as the behavior at the front is concerned, the main result in this direction is that traveling 
waves of speed $c=2$ exist~\cite{FangZhao, Gourley} and solutions to the Cauchy problem with compact initial data or which 
satisfy~\eqref{eq:u_0} propagate with speed $c(t)=2 + o(1)$~as~$t\to+\infty$~\cite{HamelRyzhik}.

As far as algebraic delays are concerned, we point to the work of Fang and Zeitouni~\cite{FangZeitouni} and Maillard and
Zeitouni~\cite{MaillardZeitouni}, as well as~\cite{NolenRyzhikRoquejoffre} where a Fisher-KPP model with a 
diffusivity that changes slowly in time was studied, and a delay, roughly,
of order $t^{1/3}$ was obtained. However, both the set-up and 
the mechanism for the large delay are quite different in these papers
than in the present work.  Finally, we also mention the recent paper of Ducrot~\cite{Ducrot} 
in which he constructs a class of non-linearities $f(x,u)$, which tend to $u(1-u)$ as $|x| \to \infty$, such that if the nonlinearity $u(1-u)$ in~(\ref{sep2704})
is replaced by~$f(x,u)$, 
then the front is at $2t - \lambda \log(t)$ for any $\lambda \geq 3/2$.

While in the final stages of preparing this paper, we learned of a very recent probabilistic study of the delay term  
by Penington~\cite{Penington}.  
In our notation, she obtains the log delay up to an error term $O(\log\log(t))$, 
when $r>3$,  and an algebraic delay $t^{(3-r)/(1+r) \pm \epsilon}$ for 
any $\epsilon>0$ for~$r \in (1,3)$.  Penington's assumptions on $\phi$ are the same as ours when $r \geq 3$. 
However, her assumptions are weaker when~$r\in (1,3)$:
the $R^{-(r-1)}$ term in~\eqref{eq:phi_upper} is replaced by $R^{-(r-1)/2}$, at the expense of a slightly less precise power in the correction. 
The proofs in~\cite{Penington} are probabilistic, involving the Feynman-Kac formula and an in-depth study of the trajectories of Brownian motion.  
Overall, our work and~\cite{Penington} are quite different and reveal different features of the equation.

\subsubsection*{Heuristics and methods of proof}

The upper bound \eqref{eq:log_delay_above} 
 is obtained by 
a rather direct adaptation of the arguments in~\cite{HNRR_Short}.  
Let us outline a heuristic argument leading to the upper 
bound~\eqref{eq:alg_delay_above} for~$r\in(1,3)$.  It also explains how 
%
%
%
the exponent ${(3-r)}/{(1+r)}$ comes about. Let the front have a delay~$d(t)$ behind $2t$, so that
\begin{equation}\label{sep2710}
    \inf_{x\leq 2t - d(t)} u(t,x) \geq \delta_0,
\end{equation}
with some $\delta_0>0$.
We expect that the solution looks like an exponential to the right of $x=2t-d(t)$ and until the ``front edge" at $x=2t+e(t)$:
\begin{equation}\label{sep2714}
u(t,x)\sim \exp\{ -(x - 2t + d(t))\},~~\hbox{for $x\in(2t - d(t),2t + e(t))$}.
\end{equation}
The diffusive Gaussian decay dominates the  
exponential ``traveling wave" decay for $x>2t+e(t)$. 
Using~\eqref{sep2710} and then~\eqref{eq:phi_lower}, one may estimate $\phi \star u(t,x)$ when $x \in (2t - d(t), 2t + e(t))$ as
\begin{equation*}
\phi \star u(t,x)
	\geq \delta_0 \int_{-\infty}^{2t - d(t)} \phi(x-y) dy
	\gtrsim (x - (2t - d(t)))^{1-r}
	\gtrsim (e(t) + d(t))^{1-r}.
\end{equation*}
Thus, in order for the exponential in (\ref{sep2714})  to be a super-solution to \eqref{eq:nonlocal} inside $(2t - d(t),2t + e(t))$, we need
\begin{equation}\label{sep2716}
(e(t) + d(t))^{1-r} \gtrsim d'(t).
\end{equation}
We also need the exponential to be above $u(t,x)$ at the front edge.   To control $u$ there, we use that, letting  $h = e^{-t} u$, $h$ is a sub-solution to the heat equation.  In other words,
\begin{equation*}
	h_t \leq h_{xx}
\end{equation*}
and, hence, for all $x \gg 1$,
\begin{equation}\label{eq:heat_subsolution}
	e^{-t}u(t,x) = h(t,x)
		\leq \int e^{-\frac{|x-y|^2}{4t}} u_0(y) dy
		\lesssim \frac{\sqrt t}{x} e^{-\frac{x^2}{4t}}.
\end{equation}
Thus, for $u$ to sit below the exponential super-solution at $x = 2t + e(t)$, we require
%
%
\begin{equation*}
\exp\Big\{t -\frac{(2t + e(t))^2}{4t} \Big\} \lesssim \exp\{-( e(t) +  d(t))\},
\end{equation*}
that is,
\begin{equation}\label{sep2720}
e(t)^2 \geq 4td(t).
\end{equation} 
Since $e(t)$ should be $o(t)$, we get 
\begin{equation}\label{sep2718}
\lim_{t \to + \infty}\frac{d(t)}{e(t)} = 0.
\end{equation}
Combining (\ref{sep2716}), (\ref{sep2720}) and (\ref{sep2718}) gives, for $t$ large,
\begin{equation*}
d'(t) \lesssim e(t)^{1-r} \lesssim t^{\frac{1-r}{2}} d(t)^{\frac{1-r}{2}},
\end{equation*}
and thus necessarily 
\[
d(t) \lesssim t^{{(3-r)}/{(1+r)}}.
\] 
We deduce also $e(t) \gtrsim t^\gamma$, with $\gamma$ as in (\ref{sep2724}).

A way to estimate the solution from below, to get the lower bounds,
is to study the linearized Fisher-KPP equation with 
a Dirichlet boundary condition at $2t + e(t)$, as in \cite{HNRR_Short}. 
The problem   that comes up after removing the exponential factor is
\begin{equation}\label{eq:self_similar_z_intro}
	\begin{aligned}
		&z_t = z_{xx} + e'(t)(z_x - z), ~~   t > 0, x > 0 ,\\
		&z(t,0) = 0.  
	\end{aligned}
\end{equation}
%
%
%
Once again, the case $r>3$ is treated similarly to \cite{HNRR_Short}. 
In particular, while the term $e'(t)z$ is important and is responsible for the $3/2$ pre-factor in the logarithmic correction, the
drift $e'(t)z_x$ is negligible. Roughly, we  estimate $z(t,x)$ at~$x \sim \sqrt{t}$, and use a ``tracing back to a shifted traveling wave'' argument, to construct a sub-solution 
for $u$. 
%
%

When $r < 3$, we choose $e(t) = t^\gamma$.  Since now $\gamma >1/2$, the drift $e'(t)z_x$ can no longer be neglected, and
the choice of the exact exponent $\gamma$ is necessary to get matching asymptotics.  
%
We explicitly construct a sub-solution of $u$ to estimate the solution at the far edge, and then perform a ``tracing back'' argument with a travelling wave. 

Lastly, in the case when $r=3$, the diffusive scale and the induced drift have the same order.   Here, the balance of these two scales causes a somewhat larger delay.

\subsubsection*{The local in time Harnack inequality}

The main tool that allows us to get ``reasonably sharp" asymptotics for the front position is a local-in-time Harnack inequality that is of an independent interest.
\begin{prop}\label{lem:Harnack}
Suppose that $u \in L^\infty([0,T]\times \R)$ is a non-negative function that solves 
\[
u_t = u_{xx} + c(t,x)u,
\]
on $[0,T]\times \R$ with $c \in L^\infty([0,T]\times \R)$ and $T>0$.  
Then, for any $p \in (1,\infty)$, there exist positive constants 
$\alpha$, $\beta$, and $C$, that depend only on $\|c\|_{L^\infty([0,T]\times \R)}$ and $p$, 
such that, for all $x, y \in \R$ and~$t \in (0,T]$, we have
\begin{equation}\label{sep2902}
u(T,x + y)
\leq C \|u\|_{L^\infty([t,T])\times \R}^{1-\frac{1}{p}} u(T,x)^{\frac{1}{p}} 
e^{\alpha t + \frac{\beta y^2}{t}}. 
\end{equation}
\end{prop}
This inequality is an indispensable tool to obtain ``reasonably sharp'' results for non-local problems.  We have used a less precise form of it to obtain the 
logarithmic delay for solutions of the cane toads equation in~\cite{BHR_Log_Delay}, and it has also been used to establish a precise lower bound on the propagation speed of solutions of a Keller-Segel-Fisher system~\cite{HamelHenderson}.
As far as we know,~\cite{BHR_Log_Delay} is the only other non-local context where a delay 
asymptotics has been established.
It allows us to bound solutions of the non-local Fisher-KPP equation
(\ref{eq:nonlocal})  in terms of the solutions of a local Fisher-KPP
equation with a local time-dependent nonlinearity~$g(t,u)$, that is logarithmic in $u$ (Gompertz type). This 
equation has inherent difficulties coming from the time dependence and the logarithmic behavior near zero, but it is much more tractable because it 
admits a comparison principle.
%
%
%


The rest of the paper is organized as follows. In  \Cref{sec:upperbound}, we present the proofs of the upper bounds \eqref{eq:log_delay_above} and \eqref{eq:alg_delay_above}.  \Cref{sec:lower_bound} is where the proofs of the lower bounds \eqref{eq:log_delay_below}, \eqref{eq:big_log_delay_below} and \eqref{eq:alg_delay_below} are given. In order to complete the proof of the lower bounds, some estimates on linearized problems with moving Dirichlet boundary conditions are obtained in \Cref{sec:estw} and \Cref{sec:estlinKPP}. 
In  \Cref{sec:local_equation}, we state and prove the result concerning the local KPP equation with logarithmic nonlinearity.
The Harnack inequality is proved in Section~\ref{sec:harnack}.

\subsection*{Acknowledgements}

The authors thank Nicolas Champagnat for the reference \cite{dennis2006estimating}.  EB was supported by ``INRIA Programme Explorateur".  LR was supported by NSF grants DMS-1311903 and DMS-1613603.
Part of this work was performed within the framework of the LABEX MILYON (ANR- 10-LABX-0070) of Universit\'e de Lyon, 
within the program ``Investissements d'Avenir'' (ANR-11- IDEX-0007) operated by the French National Research Agency (ANR). 
In addition, CH has received funding from the European Research Council (ERC) under the European Unions Horizon 2020 research and 
innovation program (grant agreement No 639638) and was partially supported by the NSF RTG grant~DMS-1246999 and DMS-1907853.

\section{Upper bounds on the location of the front}\label{sec:upperbound}

In this section, we prove the upper bounds  \eqref{eq:log_delay_above} 
and~\eqref{eq:alg_delay_above} in \Cref{thm:main_delay}. 

\subsection{The upper bound when $r > 3$}

The case $r > 3$ is very close to the local Fisher-KPP equation. The $({3}/{2}) \log t$ delay is the best case scenario -- in fact,
the delay has to be at least that large for any $r$, so the bound is a quite straightforward application of bounds  
obtained in~\cite{HNRR_Short}.

\begin{proof}[Proof of \eqref{eq:log_delay_above}]

Take $t_0 > 0$ to be determined later. Working in the moving frame with the logarithmic correction, the function 
\begin{equation*}
u_{\rm mov}(t,x) = u \Big(t, 2t - \frac{3}{2} \log\Big(1+ \frac{t}{t_0} \Big) + x\Big), 
\end{equation*}
satisfies 
\[\begin{aligned}
&(u_{\rm mov})_t \leq \Big( 2 - \frac32 \frac{1}{t+t_0} \Big)(u_{\rm mov})_x 
+ (u_{\rm mov})_{xx} + u_{\rm mov}, \qquad &&\text{ for all } t>0, x\in\R,\\
  &  u_{\rm mov}(0,x) = u_0(x), &&\text{ for all } x \in \R.
\end{aligned}\]
We construct a super-solution $\overline u$ as in~\cite{HNRR_Short}. 
Let $\overline v$ be the solution to the boundary value problem
\[\begin{aligned}
 &   \overline v_t = \Big( 2 - \frac32 \frac{1}{t+t_0} \Big)\overline v_{x} + \overline v_{xx} + \overline v,  \qquad &&\text{  for all } t>0 \text{ and } x>0, \\
   &     \overline v (t, 0) = 0,  &&\text{ for all } t>0,\\
     &    \overline v(0,x) = \1_{(0,2)}(x) &&\text{ for all } x > 0.
%
\end{aligned}\]
Then \cite[Lemma~2.1]{HNRR_Short} implies that, 
provided that $t_0$ is sufficiently large, there exists $A_0\geq 1$ 
such that for all $t \geq 0$, we have  
\begin{equation*}
\overline v(t,1) \geq A_0^{-1}.
\end{equation*}
We also have the following uniform bound on the solutions to (\ref{eq:nonlocal}).
\begin{lemma}\cite[Theorem~1.2]{HamelRyzhik}\label{lem:upper_bound}
Suppose that $u$ satisfies~\eqref{eq:nonlocal} with initial data $u_0$ 
satisfying~\eqref{eq:u_0}.  Then there exists $M>0$ such that,  
$u(t,x) \leq M$ for all $t > 0$ and $x \in \R$.
\end{lemma}
Let us now define $\bar u(t,x)$ as 
\begin{equation*}
\overline u(t,x) = M   \Big( \textbf{1}_{x \leq x_0} + \min\Big( 1 , A_0\overline v(t,x-x_0+1) \Big) \textbf{1}_{x \geq x_0}\Big),
\end{equation*}
where $M$ is as in \Cref{lem:upper_bound}.
By construction, $\overline u$ is a super-solution to $u_{\rm mov}$, and by our assumptions 
on $u_0$~\eqref{eq:u_0}, we also have $\overline u(0,x) \geq u_{\rm mov}(0,x)$ for 
all~$x\in\R$. In addition, \cite[Lemma~2.1]{HNRR_Short} implies that there exists
$T_0$ such that, for all $z$ and all $t \geq T_0$,
\begin{equation}\label{eq:chris5}
    \overline v(t,z) \leq A_0 z e^{-z}.
\end{equation}

We are now in a position to conclude the proof.  Indeed, as 
$u \leq \overline u$,
the upper bound in~\eqref{eq:chris5} implies
\begin{equation}\label{eq:c101}
\begin{split}
   	\limsup_{L\to\infty} \limsup_{t\to\infty}&\sup_{x\geq L} u\Big(t,2t - \frac32\log t + x\Big)
        = \limsup_{L\to\infty} \limsup_{t\to\infty}\sup_{x\geq L} u_{\rm mov}(t,x)\\
        &\leq \limsup_{L\to\infty} \limsup_{t\to\infty}\sup_{x\geq L} \overline u(t,x)
        \leq \lim_{L\to\infty} MA_0 L e^{-L} = 0,
\end{split}
\end{equation}
%
which concludes the proof.
%
%
\end{proof}

\subsection{The upper bound when $r = 3$}

In this section, we show how to derive the upper bound on the location of the front assuming the lower bound on the location of the front.  In other words, we prove~\eqref{eq:big_log_delay_above} assuming~\eqref{eq:big_log_delay_below}, which we prove in the next section.

\begin{proof}[Proof of~\eqref{eq:big_log_delay_above} assuming~\eqref{eq:big_log_delay_below}]

Our proof proceeds similarly as in the previous subsection.  Set $\tilde s_\phi < S_\phi$ to be determined.  Using~\eqref{eq:big_log_delay_below} and~\eqref{eq:phi_lower}, we find $L>0$ such that, for all $x\geq 0$ and $t\geq L$,
\begin{equation}\label{eq:phi_star_bound_r3}
	\phi \star u(t,x+2t - \tilde s_\phi \log(t))
		\geq \frac{1}{LA_\phi} \left(x + (S_\phi - \tilde s_\phi)\log(t) + L\right)^{-2}.
\end{equation}

Next, we use the following result that is proved in \Cref{sec:lemestpert}.
\begin{lemma}\label{lem:r3_supersolution}
	There exists $v$, $\tilde s_\phi > 3/2$, and $L$ such that
	\begin{equation}\label{eq:edge_supersolution}
	\begin{aligned}
		&v_t \geq v_{xx} + v(1 - \nu(t,x - (2t - \tilde s_\phi\log(t+t_0))), \qquad &t > L, x > 2t - s_\phi \log(t+t_0) + L,\\
		&v(L,x) \geq u(L,x), &x > 2L - \tilde s_\phi \log(L+t_0) + L,
	\end{aligned}
\end{equation}
	$v(t, L+2t - \tilde s_\phi \log(t)) \geq M+1$ for all $t \geq L$, and $v(t,x + 2t - \tilde s_\phi\log(t)) \to 0$ as $x\to\infty$ uniformly in $t\geq L$.
\end{lemma}

With \Cref{lem:r3_supersolution} in hand, we now conclude.  Notice that, \eqref{eq:phi_star_bound_r3} implies that $v$ is a super-solution of $u$ in $\{(t,x) \in [L,\infty)\times \R : x \geq 2t - \tilde s_\phi\log(t) + L\}$.  Let
\[
	\overline u(t,x)
		= \begin{cases}
			M+1, \qquad &\text{ if } x \leq 2t - s_\phi\log(t) + L,\\
			\min\{M+1,v(t,x)\} \qquad &\text{ if } x \geq 2t - s_\phi\log(t) + L.
		\end{cases}
\]
As in the previous case $r>3$, the comparison principle implies that $\overline u \geq u$ on $[L,\infty)\times\R$.  The result then follows taking $s_\phi \in (\tilde s_\phi, 3/2)$. 

\end{proof}

\subsection{The upper bound when $r \in (1,3)$}

In this section, we show how to derive the upper bound on the location of the front from the lower bound on the location of the front.  In other words, we prove \eqref{eq:alg_delay_above} assuming \eqref{eq:alg_delay_below}, which we prove in the next section.

\begin{proof}[Proof of the upper bound \eqref{eq:alg_delay_above} assuming the lower bound~\eqref{eq:alg_delay_below}]

Note that, by~\eqref{eq:c101}, we have
\begin{equation*}
\lim_{t\to\infty} \sup_{x \geq 2t + t^\gamma} u(t,x) = 0.
\end{equation*}
As a consequence, taking into account (\ref{sep2724}), it suffices to show that
\begin{equation*}
\lim_{t\to\infty} \sup_{x \in (2t - c_\phi t^{{2\gamma-1}},2t + t^\gamma ) } u(t,x) = 0.
\end{equation*}
We do this by creating a relevant super-solution to $u$ on the interval $(2t - c_\phi t^{2\gamma-1}, 2t + t^\gamma)$. 
Note that the constant $c_\phi$ is still to be determined at this stage.
Define, for any $T>0$ and $C_\phi$ as in~\eqref{eq:alg_delay_below}, the space-time domain (recall that $\gamma>1/2$ for~$1<r<3$):
\[
    \cP_{T} := \Big\{(t,x): t \in (T, \infty), x \in (2t - C_\phi t^{2\gamma-1}, 2t + t^\gamma)\Big\},
\]
and, for $(t,x) \in  \cP_{T}$, the function
\[
    \overline v(t,x) :=
        B\exp\Big\{ - \Big(x - 2t + 2c_\phi t^{2\gamma-1}\Big) \Big\}.
\]
On $\cP_{T}$, the function $\overline v$ satisfies
\begin{equation}\label{oct102}
    \overline v_t = \overline v_{xx} + \overline v \Big(1 - 2c_\phi (2\gamma -1)t^{\gamma(1-r)}\Big).
\end{equation}

The rest of the proof is devoted to showing that $u$ is, indeed, a subsolution to (\ref{oct102}) when the various constants above are suitably chosen: specifically, we show that
\begin{equation}\label{oct104}
u_t - u_{xx} - u(1 - 2c_\phi(2\gamma-1)t^{\gamma(1-r)})\le 0\hbox{ in 	$\cP_{T}$,}
\end{equation}
and
\begin{equation}\label{oct106}
u(t,x)\le \overline v(t,x),~~\hbox{ on $\partial\cP_{T}$.}
\end{equation}

First, we show that (\ref{oct104}) holds. It follows from  \eqref{eq:alg_delay_below} that
there exist $C_\phi$ and $\delta_\phi$, depending only on $\phi$, and $T_0$ such that, for all $t \geq T_0$,
\begin{equation}\label{eq:chris1}
    \inf_{x\leq 2t - C_\phi t^{{2\gamma-1}}} u(t,x) \geq \delta_\phi.
\end{equation}
Using~\eqref{eq:chris1}, we can estimate $\phi\star u$ from below,  for  $t \geq T_0$ and $x > 2t - C_\phi t^{2\gamma-1}$:
\begin{equation}\label{eq:chris3}
\begin{split}
    \phi \star u(t,x)
        &= \int_\R \phi(x-y) u(t,y) \, dy
        \geq \int_{-\infty}^{2t - C_\phi t^{2\gamma-1}} \phi(x-y) u(t,y) \, dy\\
        &\geq \delta_\phi \int_{-\infty}^{2t - C_\phi t^{2\gamma-1}} \phi(x-y) \, dy
        = \delta_\phi \int_{x - 2t + C_\phi t^{2\gamma-1}}^{+\infty} \phi(z) \, dz  \\   
        &\geq \delta_\phi A_\phi^{-1} \int_{x - 2t + C_\phi t^{2\gamma-1}}^{+\infty} z^{-r} \, dz
        = \frac{\delta_\phi}{A_\phi (r - 1)} \Big(x - 2t + C_\phi t^{2\gamma-1}\Big)^{1-r}.
\end{split}
\end{equation}
Note that, as $r>1$, we have
\[
	2\gamma-1
	=\frac{3-r}{1+r}
	= \gamma + \frac{1-r}{1+r} < \gamma.
\]
Further increasing $T$, if necessary, the right-hand side in (\ref{eq:chris3}) can be estimated, for $t\ge T$, as 
\begin{equation}\label{eq:chris3a}
\begin{aligned}
\frac{\delta_\phi}{A_\phi (r - 1)} &\Big(x - 2t + C_\phi t^{2\gamma-1}\Big)^{1-r}
	\geq \frac{\delta_\phi}{A_\phi (r - 1)} \Big(t^\gamma + C_\phi t^{2\gamma-1} \Big)^{1-r}\\
&= \frac{\delta_\phi}{A_\phi (r - 1)} \Big( 1+ C_\phi t^\frac{1-r}{1+r} \Big)^{1-r} t^{(1-r)\gamma} 
\geq \frac{\delta_\phi}{A_\phi (r - 1)} \Big( 1+ C_\phi T^\frac{1-r}{1+r} \Big)^{1-r} t^{(1-r)\gamma} \\
& \geq \frac{2^{1-r} \delta_\phi}{A_\phi (r - 1)}t^{(1-r)\gamma} \geq 2c_\phi (2\gamma-1) t^{\gamma(1-r)},
\end{aligned}
\end{equation}
as long as $c_\phi$ is sufficiently small. Now, (\ref{oct104}) follows from (\ref{eq:nonlocal}), \eqref{eq:chris3} and~\eqref{eq:chris3a}. 
%
%
%
%
%
%
%
%
%
%
%
%

To show (\ref{oct106}), first, we consider the right spatial boundary $x=2t+t^\gamma$, $t\ge T$.  
As this point is at the far edge of the front, it is natural to use the linearized problem 
\[
\begin{aligned}
   & \overline u_t = \overline u_{xx} + \overline u,~~t>0, x\in\R,\\
   & \overline u(t=0,\cdot) = u_0.
\end{aligned}
\]
Then, with $x_0$ as in \eqref{eq:u_0}, we can write for $t\geq T$:
\begin{equation}\label{eq:chris2}
\begin{aligned}
    u(t,2t + t^\gamma)
        &\leq \overline u(t,2t + t^\gamma)
        = \frac{e^t}{\sqrt{ 4 \pi t}} \int_\R e^{-\frac{(2t + t^\gamma-y)^2}{4t}} u_0(y)dy
        \leq \frac{e^t}{\sqrt{ 4 \pi t}} \int_{-\infty}^{x_0} e^{-\frac{(2t + t^\gamma-y)^2}{4t}} dy\\
        &= \frac{e^t}{\sqrt{\pi}} \int_{\frac{2t + t^\gamma-x_0}{2\sqrt{t}}}^{+\infty} e^{-y^2} dy
        \leq \frac{Ce^t \sqrt {t}}{2t + t^\gamma - x_0} e^{-\frac{(2t + t^\gamma - x_0)^2}{4t}} 
        \\
        &\leq C_0 \exp\Big\{- t^{\gamma} - \frac14 t^{2\gamma -1}\Big\} 
        \leq B \exp\Big\{- t^\gamma - 2c_\phi t^{2\gamma-1}\Big\} =  \overline v(t,2t + t^\gamma),
\end{aligned}
\end{equation}
so long as $B \geq C_0$. Above, we have increased $T$ and decreased $c_\phi$ if necessary.  The constant $C_0$ depends only on $\gamma$ and~$x_0$.
%
Thus,  (\ref{oct106}) holds at $x=2t+t^\gamma$ for all $t\geq T$ as long as $B \geq C_0$.

At the left boundary $ x = 2t - C_\phi t^{2\gamma-1} $, we have
\begin{equation}\label{eq:chris8}
\overline v(t,2t - C_\phi t^{2\gamma-1})= B \exp\Big\{ (C_\phi - 2c_\phi) t^{2\gamma-1} \Big\} 
 \geq M \geq u\Big(t, 2t - C_\phi t^{2\gamma-1}\Big),
\end{equation}
as long as $2c_\phi \leq C_\phi$ and $B \geq M$. Here, $M$ is the upper bound in \Cref{lem:upper_bound}.

Lastly, we check that (\ref{oct106}) holds at $t=T$, for $2T - C_\phi T^{2\gamma-1}\le x\le 2T + T^\gamma$: 
\[    \overline v(T, x)
        = B \exp \Big\{ - \Big(x - 2T + 2 c_\phi T^{2\gamma-1}\Big)\Big\} \geq B \exp \Big\{ - T^\gamma - 2c_\phi T^{2\gamma-1}\Big\}.\]
As long as $B \geq M \exp\Big\{T^\gamma + 2c_\phi T^{2\gamma-1}\Big\}$, we have that, for all $x \in [2T - C_\phi T^{2\gamma-1}, 2T + T^\gamma]$
\begin{equation}\label{eq:chris9}
    \overline v(T, x)
        \geq M \geq u(T, x),
\end{equation}
and (\ref{oct106}) holds on all of $\partial\cP_T$. 

It follows from (\ref{oct104}) and (\ref{oct106}) that, with $T$ and $B$ sufficently large, and $c_\phi$ sufficiently small, we have
%
%
%
%
%
%
%
%
%
\[
    \lim_{t\to\infty} \sup_{x \geq 2t - c_\phi t^{2\gamma - 1}} u(t,x)
        \leq \lim_{t\to\infty} \sup_{x \geq 2t - c_\phi t^{2\gamma - 1}} \overline v(t,x)
        \leq \lim_{t\to\infty} B \exp\Big\{- \Big(2c_\phi -  c_\phi \Big) t^{2\gamma -1}\Big\} = 0,
\]
which finishes the proof of the upper bound.
\end{proof}

\section{Lower bounds on the location of the front}\label{sec:lower_bound}



The proofs of the lower bounds in \Cref{thm:main_delay} are much more involved. They hinge on estimating~$\phi \star u$ in terms 
of  $u$ in a local way, 
and then deriving precise heat kernel type estimates on the resulting local  equation. 

\subsection{Estimating the non-local term by a local counterpart}

To begin, we estimate the convolution term $\phi \star u$ in terms of $u$ under the assumptions of \Cref{thm:main_delay}.  
The assumptions of these two theorems differ only in the range of $r$.  In this section, we assume only that $r>1$ so our 
computations apply to all cases.

\begin{lemma}\label{lem:estphistaru}
There exists $C_{\rm conv}>0$, depending only on $\phi$, such that, for all $t \geq 1$ and all $x \in \R$,
\begin{equation}\label{eq:phi_u}
    \phi \star u(t,x)
        \leq C_{\rm conv} \max\Big\{1, \Big[ \frac{1}{t} \log\Big( \frac{M}{u(t,x)}\Big)\Big]^{\frac{r-1}{2}}\Big\} \log\left(\frac{M}{u(t,x)}\right)^{1-r}. 
\end{equation}
\end{lemma}

\begin{proof}
It is here that the local-in-time Harnack inequality is used crucially.  
Fix any time $t\geq 1$ and~$x,y\in\Rm$.   \Cref{lem:Harnack} 
with $p=2$ implies that there exists $\alpha>0$ so that  
\begin{equation}\label{eq:Harnack_application}
    u(t,x+y)
        \leq C \sqrt{u(t,x)} \exp\Big\{\alpha t' + \frac{\alpha y^2}{t'}\Big\},~~\hbox{ for all $t' \in (0,t]$, }.
\end{equation}
Above, we absorbed the uniform bound $M$ of $\|u\|_\infty$ given by \Cref{lem:upper_bound} into the constant $C$.  By increasing $M$ if necessary, we may assume that $M \geq \|u\|_\infty + 1$, which allows us to simplify notation in the sequel.  Using~(\ref{eq:phi_upper})
and~\eqref{eq:Harnack_application}, we obtain, for $R>0$ and $t' \in (0,t]$ to be determined,
\begin{equation}\label{oct110}
\begin{split}
    \phi \star u(t,x)
        &\leq \int_\R \phi(y) u(t,x-y) \, dy
        \leq C \int_{B_R} \phi(y) \sqrt{u(t,x)} e^{\alpha t' + \frac{\alpha R^2}{t'}} dy
            + M\int_{B_R^c} \phi(y)  dy\\
        &\leq  C \sqrt{u(t,x)}R \exp\Big\{\alpha t' + \frac{\alpha R^2}{t'}\Big\} + C M R^{-r + 1}.
\end{split}
\end{equation}
The constant $C$ changes line-by-line for the remainder of the proof and depends only on $\phi$ and $\alpha$. 

We now optimize the right-hand side in (\ref{oct110}) 
with respect to $t' \in (0,t]$ and $R > 0$. 
If $t' = R$, then
\begin{equation}\label{eq:phi_u1}
\begin{split}
    \phi \star u(t,x)
    	\leq C \sqrt{u(t,x)} R e^{2\alpha R} + C M R^{-r + 1}.
\end{split}
\end{equation}
To roughly balance the two terms in the right side of (\ref{eq:phi_u1}),
we choose 
\begin{equation}\label{oct112}
R = \frac{1}{8\alpha}\log\Big(\frac{M}{u(t,x)}\Big),
\end{equation}
the most important point being that $R$ should be of order $\log u$. As we have set $t'=R$ in (\ref{eq:phi_u1}), 
and we need to have $0\le t'\le t$, the choice (\ref{oct112}) is possible
only if
\begin{equation}\label{oct114}
t\geq \frac{1}{8\alpha}\log\Big(\frac{M}{u(t,x)}\Big).
\end{equation}
With this, we find, from (\ref{eq:phi_u1}): 
\begin{equation}\label{eq:phi_u1bis}
\begin{aligned}
    \phi \star u(t,x)
        &\leq C  \sqrt{u(t,x)} \log\Big(\frac{M}{u(t,x)}\Big) \exp\Big\{- \frac14 \log\Big(\frac{u(t,x)}{M}\Big) \Big\} 
        + C\Big(\log\Big(\frac{M}{u(t,x)}\Big)\Big)^{1-r}\\
        &\leq C\Big( u(t,x)^{1/4} \Big(\log\Big(\frac{M}{u(t,x)}\Big)\Big)^r + 1 \Big) \Big(\log\Big(\frac{M}{u(t,x)}\Big)\Big)^{1-r}
        \leq C \Big(\log\Big(\frac{M}{u(t,x)}\Big)\Big)^{1-r}.
\end{aligned}
\end{equation}
%
When (\ref{oct114}) does not hold, so that
\begin{equation}\label{oct116}
t\leq \frac{1}{8\alpha}\log\Big(\frac{M}{u(t,x)}\Big).
\end{equation}
we choose $t' = t$ and set 
\[
R = \Big(\frac{t}{8\alpha}\log\Big(\frac{M}{u(t,x)}\Big)\Big)^{1/2},
\]
in (\ref{oct110}),  leading to
\begin{equation}\label{eq:phi_u2}
\begin{split}
    \phi&*u(t,x)
        \leq C \sqrt{u(t,x)} \sqrt{t} \log^{1/2}\Big(\frac{M}{u(t,x)}\Big) 
        \exp\Big\{\alpha t + \frac18\log\Big(\frac{M}{u(t,x)}\Big)\Big\} + C\Big(t \log\Big(\frac{M}{u(t,x)}\Big)\Big)^{-\frac{(r-1)}{2}}\\
        &\leq C\sqrt{u(t,x)} \log\Big(\frac{M}{u(t,x)}\Big) \exp\Big\{\frac14\log\Big(\frac{M}{u(t,x)}\Big) \Big\} + C\Big(t \log\Big(\frac{M}{u(t,x)}\Big)\Big)^\frac{-(r-1)}{2} \\
        &\leq C  u(t,x)^{1/4}\log\Big(\frac{M}{u(t,x)}\Big) + C\Big(t \log\Big(\frac{M}{u(t,x)}\Big)\Big)^\frac{-(r-1)}{2}  \\
	& \leq C   \Big(1+u(t,x)^{1/4}\log\Big(\frac{M}{u(t,x)}\Big)\Big(t \log\Big(\frac{M}{u(t,x)}\Big)\Big)^\frac{r-1}{2}  \Big) 
	\Big(t \log\Big(\frac{M}{u(t,x)}\Big)\Big)^\frac{-(r-1)}{2}\\
        &\leq C \Big( u(t,x)^{1/4}\Big( \log\Big(\frac{M}{u(t,x)}\Big)\Big)^r + 1\Big) \Big(t \log\Big(\frac{M}{u(t,x)}\Big)\Big)^\frac{-(r-1)}{2}
        \leq C \Big(t \log\Big(\frac{M}{u(t,x)}\Big)\Big)^\frac{1-r}{2}.
\end{split}
\end{equation}
We used  (\ref{oct116}) several times above, as well as the upper bound $u(t,x)\le M$ in the last inequality.
The combination of~\eqref{eq:phi_u1bis} and~\eqref{eq:phi_u2} concludes the proof of the lemma.

\end{proof}

\subsection{A local equation and related bounds}

In view of \Cref{lem:estphistaru}, it is natural to introduce the following nonlinearity. Fix $r>1$, and 
for any positive constants $\theta_g$ and $A_g$, set $\Theta_g := \theta_g\exp\{-A_g^{1/(r-1)}\}$ and define $g \in C^{0,1}$ on
$(0,\Theta_g)$ as 
\begin{equation}\label{eq:g}
    g(t,u) :=  
        A_g \max\Big\{1, \Big[  \Big(t+A_g^\frac{1}{r-1}\Big)^{-1} \log\Big( \frac{\theta_g}{u }\Big)\Big]^{\frac{r-1}{2}}\Big\} 
        \log\Big(\frac{\theta_g}{u }\Big)^{1-r},~~\text{if } u \in\Big(0,\Theta_g \Big).
 \end{equation}
Outside  $[0,\Theta_g]$ we set $g(t,u)=0$ for $u<0$ and $g(t,u)=1$ for $u>\Theta_g$.
By construction,~$g(t,\cdot)$ is continuous. The ``$A_g^{1/(r-1)}$'' term 
in the second part of the maximum in the definition of $g$ does not affect the analysis in any way.  In fact, any other choice of $g$ that preserves the asymptotics as~$u$ and $t$ tend to zero would have the desired properties that we prove in the sequel.

We will make use of the  local   equation with a moving boundary at the front edge:
 \begin{equation}\label{eq:w}
 \begin{aligned}
 &w_t = w_{xx} + w(1 - g(t,w)),  &&\text{ in } \cP_{g,\gamma} := \Big\{(t,x) : t> 0, x > 2t + (t+t_0)^\gamma - t_0^\gamma \Big\},\\
 & w(t, 2t + (t+t_0)^\gamma - t_0^\gamma) = 0, \quad &&\text{ for all } t > 0,\\
 &       w(0,x) = w_0(x) && \text{ for all } x > 0.
    \end{aligned}
\end{equation}
The following proposition contains the crucial lower bounds for the solutions of (\ref{eq:w}) we will need. 
\begin{prop}\label{prop:g}
Assume that there exists $\delta_w >0$ and $x_w \in \R^+$ such that  the initial condition~$w_0(x)$ for (\ref{eq:w})
satisfies $w_0(x)\ge \delta_w 1_{(0,x_{w})}(x)  $.
\begin{enumerate}  
\item     If $r>3$, then there exists $X_w$ and $T_0$ such that if $x_w \geq X_w$ and $t_0 \geq T_0$ then there exists a positive constant $B_1$, depending only on $x_w$, $\delta_w$, $t_0$, $\gamma$, and $g$, such that, for all $t$ sufficiently large, we have
    \[
        w(t,2t + t^\gamma + \sqrt t ) \geq B_1^{-1} t^{-1}  e^{- \sqrt t - t^\gamma}.
    \]
\item  If $r=3$, then set $t_0 = 1$. There exists $\bar N>0$ such that if $x_w \geq 1$ then there exists a positive constant $B_2$, depending only on $\bar N$, $\delta_w$, $\gamma$, and $g$, such that, for all $t$ sufficiently large, we have
    \[
    	w(t,2t + 2 \sqrt t ) \geq B_2^{-1} t^{-1 - \bar N}e^{- \sqrt t}.
    \]
\item    If $r \in (1,3)$, then set $t_0 = 1$.  There exists $B_3 > 0$, depending only on $\delta_w$ and $g$, such that if~$x_w \geq 1$ then, for all $t \geq 1$, we have
    \[
		w(t,2t + t^\gamma + \sqrt t) \geq B_3^{-1} e^{- \sqrt t - t^\gamma - B_3 t^{2\gamma-1}}.
    \]

 \end{enumerate}   
    
\end{prop}
We delay the proof of this proposition until \Cref{sec:estw} and now continue the proof of the lower bounds of \Cref{thm:main_delay}. Having reduced the problem to estimating a delay for a local equation, we now transfer known bounds of \Cref{thm:main_delay} on $w$ to bounds on $u$.

\subsection{From a bound on $w$ to a bound on $u$}

Let us take $\theta_g = M$ and $A_g = C_{\rm conv}$ in the definition~\eqref{eq:g} of $g(t,u)$ and let the initial condition in~(\ref{eq:w})
be $w_0(x)= e^{-M } u_0(x)$.  A combination of   Lemma~\ref{lem:estphistaru} and \Cref{prop:g} 
implies that~$u$ is a super-solution for $w$ for $t\ge 1$.  Further, it follows from  considerations as in~\cite[Section~3]{BHR_Log_Delay}, 
that~$w(1,x) \leq u(1,x)$ for all $x\in \R$ due to the $e^{-M}$ pre-factor in the definition of $w_0$. 
The maximum principle then implies that $w(t,x) \leq u(t,x)$ for all $t \geq 1$ and all $x\in \R$.

Using the assumptions on the initial data~\eqref{eq:u_0}, we can, up to translating $u_0$, and thus $w_0$ as well, assume that $w_0$ satisfies the hypothesis $x_w = x_0 \geq X_w$ in \Cref{prop:g}.  Translating further and using parabolic regularity we may remove the dependence on $t_0$.  As a direct consequence, we have established:

\begin{corollary}\label{prop:estu}
    Suppose that $u$ satisfies~\eqref{eq:nonlocal} and~\eqref{eq:u_0} with $\phi$ satisfying~\eqref{eq:phi1},~\eqref{eq:phi2}, and~\eqref{eq:phi_upper}.  Then there exists~$S_0$, depending only on $u_0$ and $\phi$, such that:
  \begin{enumerate}  
\item     If $r>3$, then there exists a positive constant $B_1$, depending only on $u_0$ and $\phi$ such that, for all $t$ sufficiently large, we have
    \[
        u(t,2t + t^\gamma + \sqrt t - S_0  ) \geq B_1^{-1} t^{-1}  e^{-\sqrt t - t^\gamma}.
    \]
\item If $r=3$, then there exist positive constants $\bar N$, and $B_2$ such that, for all $t$ sufficiently large, we have
	\[
		u(t, 2t + 2\sqrt t - S_0)
			\geq  B_2^{-1} t^{-\left(1 + \bar N\right)} e^{-2 \sqrt t}.
	\]
\item    If $r \in (1,3)$, then there exists a positive constant $B_3$, depending only on $u_0$ and $\phi$, such that, for all $t \geq 1$, we have
    \[
	u(t,2t + t^\gamma + \sqrt t - S_0) \geq B_3^{-1} e^{- t^\gamma - B_3 t^{2\gamma-1}}.
    \]
 \end{enumerate}   
\end{corollary}

\subsection{From a bound on $u$ on the right to the location of the front}

We are now in a position to obtain the lower bounds~\eqref{eq:log_delay_below},~\eqref{eq:big_log_delay_below}, 
and~\eqref{eq:alg_delay_below}. Thanks to  \Cref{prop:estu}, we fit a suitable translate of a traveling wave solution for (\ref{eq:w})
underneath $u$, for $x\leq 2t + t^\gamma + \sqrt t$.

Fix any $A_V> C_{\rm conv}$ and let $V$ be a traveling wave solution of
\[
    - 2 V' = V'' + V\Big(1 - A_V\log\Big(\frac{M}{V}\Big)^{1-r}\Big),~~\hbox{
    $V(-\infty) = M \exp\Big\{-A_V^{1/(r-1)}\Big\}$ and $V(+\infty) = 0$.}
\]
The existence, uniqueness up to translation, and monotonicity of $V$ is given by, for example,~\cite{BerestyckiNirenberg}.  We also recall (see~\cite{Hamel,BouinHendersonForthcoming}) the fact that there exists $\beta_{r,A_v}>0$, depending only on $r$ and $A_V$, and $\kappa > 0$ such that, as $\xi\to\infty$
\begin{equation}\label{oct202}
	V(\xi) \sim \begin{cases}
			\kappa \xi e^{-\xi}, \qquad &\text{ if } r > 3,\\
			\kappa \xi^{1+\beta_{r=3,A_V}} e^{-\xi}, \qquad &\text{ if } r = 3,
		\end{cases}
\end{equation}
and, if $r \in (1,3)$,
\begin{equation}\label{eq:tw_asymptotics2}
	\frac{\log(V(\xi))+\xi}{\xi^{(3-r)/2}}
		\sim \beta_{r,A_V}.
\end{equation}
Define $v$ as
\begin{align*}
    &v(t,x) = V\left(x - 2t + \frac32 \log t + s_0\right), &\text{ if } r>3,\\
    &v(t,x) = V\left(x - 2t + \Big(\frac32 + \bar N + \frac{\beta_{r=3,A_V}}{2}\Big)\log t + s_0\right), &\text{ if } r = 3,\\
    &v(t,x) = V\left(x - 2t + (2B_3+\beta_{r,A_V}) t^{2\gamma-1} + s_0\right), &\text{ if } r \in (1,3),
\end{align*}
where the shift $s_0$ is to be determined below and $\lambda$ is as in  \Cref{prop:g} and \Cref{prop:estu}.


\begin{lemma}\label{lem:compTW}
There exists $T_1 > 0$ and $\underline s_0$ such that if $s_0 \geq \underline s_0$, then $v(t,x) \leq u(t,x)$ if $r\neq 3$, $t \geq T_1$, and $x\leq 2t + t^\gamma + \sqrt t - S_0$ where $S_0$ is the shift given in \Cref{prop:estu}.
\end{lemma}

\begin{proof}
We prove the lemma for $r>3$, so that $\gamma<1/2$,
the proof being the same in the other cases up to situational modifications. We use the parabolic maximum 
principle. First, we note that, up to increasing $s_0$ and $A_V$, we may ensure that 
\[
v(T_1,x) \leq u(T_1,x)~~\hbox{ for 
all $x\leq 2T_1 + \sqrt T_1 + T_1^\gamma - S_0$.}
\]
Second, we claim that, up to increasing $s_0$, we have
\[
v(t,2t + \sqrt t + t^\gamma-S_0) \leq u(t,2t + \sqrt t + t^\gamma-S_0)~~\hbox{ for all $t \geq T_1$.}
\]
Indeed, for $t$ sufficiently large, (\ref{oct202}) implies, as $\gamma<1/2$:
\begin{eqnarray}\label{eq:chris13}
    &&\!\!\!\!\!\!\!\!\!\!\!\!\!\!\!\!\!\!\!\!\!\!\!\!
    v(t,2t + \sqrt{t} + t^\gamma-S_0)
        = V\Big( \sqrt{t} + t^\gamma + \frac32\log t + s_0 -S_0\Big)\\
        &&~~~~~~~\leq 2\kappa \Big(\sqrt{t} + t^\gamma + \frac32\log t + 
        s_0- S_0\Big)\exp\Big\{- \Big(\sqrt{t} + t^\gamma + \frac32\log t + s_0 - S_0 \Big)\Big\}\nonumber\\
        &&~~~~~~~\leq 4\kappa \sqrt{t} t^{-3/2} \exp\Big\{- \sqrt{t} - t^\gamma - s_0 + S_0\Big\} \leq 4\kappa t^{-1} 
        \exp\Big\{- \sqrt{t} - t^\gamma - s_0 + S_0\Big\}.\nonumber
\end{eqnarray}
It follows that
\[\begin{split}
    v(t,2t + \sqrt{t} + t^\gamma - S_0)
        \leq 4\kappa  e^{S_0 - s_0}  B_1  u(t,2t + \sqrt{t} + t^\gamma - S_0) 
        \leq u(t,2t + \sqrt{t} + t^\gamma - S_0 ),
\end{split}\]
%
for $T_1$ sufficiently large and all $s_0 \geq S_0 + \log(4\kappa B_1)$.
%
%

Third, up to increasing $A_V$, the ordering holds true near $-\infty$.  Indeed, using~\Cref{lem:estphistaru} and 
the assumptions~\eqref{eq:u_0} on $u_0$, it is easy to see that there exists $\delta>0$, depending only on $u_0$ and $\phi$ such that, 
for any $\overline x<0$ with $|\overline x|$ is sufficiently large, the function
\[
	\underline u(x) = \delta \cos((x- \overline x)/100)
\]
is a sub-solution for $u$ for all $t\geq 1$, so that $\delta = \underline u(\overline x) \leq u(t,\overline x)$ for all $t\geq 1$.
Thus, increasing $A_V$, if necessary, we have that, for all $t> 0$,
\[
	\lim_{x\to-\infty} v(t,x) < M e^{-A_V^{1/(r-1)}}
		< \delta
		\leq \inf_{t\geq 1}\liminf_{x\to-\infty} u(t,x).
\]
Now, assume for the sake of a contradiction that there exists a first touching time $(t_{\rm ft}, x_{\rm ft})$ such that 
\[
\hbox{$t_{\rm ft} \geq T_1$, $x_{\rm ft} \leq 2t_{\rm ft} + \sqrt{t_{\rm ft}} + t_{\rm ft}^\gamma - S_0$,}
\]
and
\[
u(t_{\rm ft}, x_{\rm ft}) = v(t_{\rm ft}, x_{\rm ft}),
\]
and $u(t,x) > v(t,x)$ for all $t \in [T_1, t_{\rm ft})$ and $x < 2t + \sqrt{t} + t^\gamma - S_0$.  Our goal is to obtain a contradiction by estimating $\phi \star u$ and looking at the equation satisfied by $u - v$.

First, we estimate $\phi \star u(t_{\rm ft}, x_{\rm ft})$ using \Cref{lem:estphistaru}.  By increasing $s_0$ if necessary, we obtain
\begin{equation}\label{eq:chris14}
\begin{split}
    v(t,2t + \sqrt{t} + t^\gamma-S_0)
        &= V\Big(\sqrt{t}  + t^\gamma + \frac32\log(t)+ s_0 -S_0\Big)\\
        &\geq \frac{\kappa}{2} \Big(\sqrt{t} + t^\gamma + \frac32\log t + s_0 - S_0\Big)
        \exp\Big\{- \Big(\sqrt{t} + t^\gamma + \frac32\log t + s_0 - S_0\Big)\Big\}\\
        &\geq \frac{\kappa}{2t^{3/2}} \left(\sqrt{t} + t^\gamma + \frac32\log t + s_0 - S_0\right)
        \exp\Big\{- \sqrt{t} - t^\gamma - s_0 + S_0\Big\}.
\end{split}
\end{equation}
Since $V$ is monotonic, $\gamma < 1$, and $x_{\rm ft} \leq 2t_{\rm ft} + t_{\rm ft}^{\frac12} + t_{\rm ft}^\gamma-S_0$, it follows that up to increasing $T_1$, we have that
\[
	u(t_{\rm ft},x_{\rm ft}) = v(t_{\rm ft},x_{\rm ft}) \geq M e^{-t_{\rm ft}},
\]
which, in turn, implies that
\[
	\log\Big(\frac{M}{u(t_{\rm ft},x_{\rm ft})} \Big)^{-(r-1)}
		\geq t_{\rm ft}^{-\frac{r-1}{2}} \log\Big(\frac{M}{u(t_{\rm ft},x_{\rm ft})}\Big)^{-\frac{r-1}{2}}.
\]
In view of the bound on $\phi \star u$ obtained in \Cref{lem:estphistaru}, we have that, at $(t_{\rm ft},x_{\rm ft})$,
\begin{equation}\label{eq:chris15}
	u_t - u_{xx} - u\Big(1 - A_V\log\Big(\frac{M}{u}\Big)^{-(r-1)}\Big)
		\geq \Big( A_V - C_{\rm conv}\Big) u\Big(\log\Big(\frac{M}{u}\Big)\Big)^{1-r}
		> 0,
\end{equation}
where we used the fact that $A_V > C_{\rm conv}$ in the last inequality.  In addition, we note that
\begin{equation}\label{eq:chris16}
\begin{split}
    v_t - &v_{xx} - v\Big(1 - A_V \log\Big(\frac{M}{v}\Big)^{-(r-1)}\Big)\\
        &= \Big(\frac{3}{2(t+1)} - 2\Big) V' - V'' - V\Big(1 - A_V \log\Big(\frac{M}{V}\Big)^{-(r-1)}\Big)
        = \frac{3}{2(t+1)} V' \leq 0.
\end{split}
\end{equation}
Hence, setting $\psi = u-v$,~\eqref{eq:chris15} and~\eqref{eq:chris16}, imply that
\[
	\psi_t - \psi_{xx} - c \psi > 0,
\]
where we define
\[
	c:= \frac{u\left(1 - A_V \log\Big(\frac{M}{u}\Big)^{-(r-1)}\right) - v\left(1 - A_V \log\Big(\frac{M}{v}\Big)^{-(r-1)}\right) }{u - v}.
\]
Notice that, due to the Lipschitz continuity of $w\mapsto  w\left(1 - A_V \log\Big(\frac{M}{w}\Big)^{-(r-1)}\right)$ on compact subsets of $[0,M)$, $c \in L^\infty$.  On the other hand, using that $t_{\rm ft}$ is the first time that $\psi$ touches zero and $x_{\rm ft}$ is the location of a minimum of $\psi$, we have that
\[
	\psi_t - \psi_{xx} - c \psi \leq 0.
\]
This yields a contradiction, finishing the proof.
\end{proof}

The lower bounds now follow easily.

\begin{proof}[Proof of~\eqref{eq:log_delay_below},~\eqref{eq:big_log_delay_below}, and \eqref{eq:alg_delay_below}]

We conclude the proof by noticing that, for all $t \geq T_1$,
\begin{equation}
\begin{split}
	\inf_{x \leq 2t - (3/2)\log t} u(t,x)
		&= \inf_{x \leq 0} u\Big(t,x + 2t - \frac32\log t\Big)
	        \geq \inf_{x \leq 0} v\Big(t,x + 2t - \frac32\log t\Big)\\
		&= \inf_{x \leq 0} V\Big(x + s_0\Big) = V(s_0), 
\end{split}		
\end{equation}
which means that \eqref{eq:log_delay_below} holds.  The proofs of \eqref{eq:alg_delay_below} and~\eqref{eq:big_log_delay_below} are similar and, thus, omitted.
%
%
%
%
%

\end{proof}

%
%
%
%

\section{Proof of \Cref{prop:g}}\label{sec:estw}

To obtain estimates on the solution of (\ref{eq:w}), we consider the corresponding linearized problem with the Dirichlet boundary condition:
\begin{equation}\label{eq:self_similar_v0}
	\begin{aligned}
		&\tilde v_t = \tilde v_{xx} + \tilde v, &&\text{ on } \{(t,x) : t > 0, x > 2t + (t + t_0)^\gamma - t_0^\gamma\},\\
		&\tilde v(t,2t + (t+t_0)^\gamma - t_0^\gamma) = 0,  \qquad &&\text{ for all } t >0,\\
		&\tilde v(0,x) = w_0(x), &&\text{ for all } x > 0,
	\end{aligned}
\end{equation}
where $w_0$ is as in \Cref{prop:g}.

\subsection{The case $r>3$}

The following key lemma about solutions to \eqref{eq:self_similar_v0}
allows us to prove \Cref{prop:g} when $r > 3$.
We prove this lemma in \Cref{sec:self_similar}.
\begin{lemma}\label{lem:self_similar}
Assume $r > 3$.  If $t_0$ and $x_w$ are sufficiently large, depending only on $\gamma$, there exist positive constants $T$ and $B$, depending only on $w_0$ and $t_0$, such that, for all $t \geq T$, we have~$\|\tilde v(t,\cdot)\|_\infty \leq B e^{ - t^\gamma}$ and
\[
	\tilde v(t,2t  + t^\gamma + \sqrt{t}) \geq B^{-1} t^{-1} \exp\Big\{ - \sqrt{t} - t^\gamma\Big\}.
\]
\end{lemma}

 We now finish the proof of \Cref{prop:g}.
Let $\tilde v$ be as in \Cref{lem:self_similar}. We may assume, without loss of generality, that $T\geq 1$, and set
\[
\delta = \min\Big\{B^{-1},B^{-1} \theta_g e^{-A_g^{1/(r-1)}}, e^{-T}\Big\}.
\] 
We also take  a continuous function $a(t) \leq 1$ for all $t\geq 0$, to be determined, and set
 \begin{equation}\label{oct212}
 \underline v(t,x) = \delta a(t) \tilde v(t,x).
 \end{equation}
Using~\eqref{eq:self_similar_v0}, we obtain
 \begin{equation}\label{eq:chris22}
 \underline v_t - \underline v_{xx} - \underline v(1 - g(t,\underline v))
		= \delta a' \tilde v + \delta a \tilde v_t - \delta a \tilde v_{xx} - \delta a \tilde v + \delta a \tilde v g(t, \delta a \tilde v)
		= \delta \tilde v\Big( a' + a g(t,\delta a \tilde v)\Big).
\end{equation}
Thus,  $\underline v$ is a sub-solution of $w$ for $t\ge T$ as long as 
\[
a' + a g(t,\delta a \tilde v) \leq 0.
\]
Using the upper bound on $\tilde v$ along with the definition of $\delta$, we see that this inequality would hold if 
\begin{equation}\label{oct210}
a' + a A_g \max\Big\{1, \Big[  \Big(t+A_g^\frac{1}{r-1}\Big)^{-1} 
\log\Big( \frac{1}{a e^{-t^\gamma}}\Big)\Big]^{\frac{r-1}{2}}\Big\} \log\Big(\frac{1}{ae^{-t^\gamma}}\Big)^{1-r} \leq 0.
\end{equation}
A lengthy but straightforward computation using, in particular, that $A_g\ge 1$,
shows that (\ref{oct210}) is satisfied if we take
\[
a(t) = \exp\Big\{\beta \Big[(t+1)^{2\gamma-1} -1\Big]\Big\},
\]
with a suitable $\beta>0$.


Hence $\underline v$ is a sub-solution of $w$.  Further, 
arguing as in~\cite[Section~3]{BHR_Log_Delay} and using the choice 
of~$\delta$ and $a$, we have that $\underline v(T,x) \leq w(T,x)$ for all $x \geq 2T + (T+t_0)^\gamma - t_0^\gamma$.  The maximum principle then implies that 
 $\underline v(t,x) \leq w(t,x)$ for all $t > T$ and $x > 2t + (t+t_0)^\gamma - t_0^\gamma$.  The conclusion of the proposition follows immediately from \Cref{lem:self_similar} since $2t + t^\gamma \geq 2t + (t+t_0)^\gamma - t_0^\gamma$.

\subsection{The case $r = 3$}

We follow here the same strategy as for $r>3$, but the estimates on $\tilde v$ are obtained differently. 

\begin{lemma}\label{lem:self_similar_critical}
For $r = 3$ and $t$ sufficiently large, there exist $\lambda$ and $B>0$ such that 
\[
	\|\tilde v(t,\cdot)\|_\infty  \
		 \leq B e^{-\sqrt t},
\]
and 
\[
	\tilde v(t,2t  + 2 \sqrt{t} ) 
		\geq B^{-1} t^{-1-\lambda} \exp\Big\{ - 2 \sqrt{t} \Big\}.
\]
\end{lemma}

With this lemma, proved in Section~\ref{sec:subsolr=3},
 one may repeat the argument for $r>3$, building a sub-solution
$\underline v(t,x)$ as in~(\ref{oct212}), with $\delta>0$ sufficiently small,
and $a(t)$ such that
\[
	a'  + a A_g \max \Big\{1, \Big(t + \sqrt A_g\Big)^{-1} \log\Big(\frac{e^{ \sqrt t}}{C\delta a}\Big) \Big\}\log\Big(\frac{e^{\sqrt t}}{C\delta a}\Big)^{-2} \leq 0.
\]
The above inequality is satisfied with $a(t) = (t+\sqrt A_g)^{-N}$ 
for all $t\geq 1$ so long as $\delta$ is chosen small enough and $N$ is chosen large enough, depending only on $A_g$ and $C$.

\subsection{The estimate when $r\in (1,3)$}

Here we directly construct a sub-solution of $w$.  
We seek a sub-solution $\tilde v$ solving
\begin{equation}\label{eq:linearized}
	\begin{aligned}
		&\tilde v_t \leq \tilde v_{xx} + \tilde v,  &&\text{ for}~ t > 0, x > 2t + (t+1)^\gamma -1,\\
		&\tilde v(t,2t + (t+1)^\gamma-1) = 0, \qquad &&\text{ for } t > 0.
	\end{aligned}
\end{equation}
%
%
%
Recall that $t_0=1$ in parts 2 and 3 of \Cref{prop:g}. Given $a>0$, set
\begin{equation}\label{eq:sub-solution_v}
	v(t,x) = \frac{x}{(1+t)^3}\exp\Big\{-x - \frac{\gamma}{2} x \Big(1+t\Big)^{ \gamma-1} - 
	\Big(1+t\Big)^\gamma - \Big[\frac{\gamma^2}{4(2\gamma - 1)} + a\Big](1+t)^{2\gamma - 1}  - \frac{x^2}{2(1+t)} \Big\}.
\end{equation}
Here, the key computation is the following:
\begin{lemma}\label{lem:sub-solution}
There exists $a_0>0$ such that if $a \geq a_0$ 
then $\tilde v(t,x) = v(t,x-(2t + (t+1)^\gamma-1))$ solves \eqref{eq:linearized}. 
\end{lemma}
We delay the proof of \Cref{lem:sub-solution} until \Cref{sec:sub-solution} and proceed with the proof of \Cref{prop:g}.

\paragraph{A bound for small times.}

Unfortunately, $v$ is not compactly supported at $t=0$, so we need to ``fit it under'' $w$ at a later time.  To do this, we first obtain a preliminary   
lower bound on~$w$ at time~$1$ by using the infinite speed of propagation of the heat equation.    
Recall that~$w_0 \geq \delta_w \1_{(-\infty,x_w)}$ and  $1 - g(t,w) \geq 0$.  
Hence, we have 
\[
w_t - w_{xx} \geq 0,
\]
so that $w$ is a super-solution to the heat equation with a Dirichlet boundary condition fixed 
at 
\[
 \overline x_0 := 2\cdot2+(2+t_0)^\gamma - t_0^\gamma = 3^\gamma +3,
\]
on the time interval $[0,2]$.  
It follows that
\begin{equation}\label{eq:w1_lower_bound}
\begin{split}
	w(2,x + \overline x_0)
		&\geq \frac{1}{\sqrt{ 8 \pi}} \int_0^\infty w_0(y {+}\overline x_0) 
		\Big[e^{-|x-y|^2/{8}} - e^{-{|x+y|^2}/{8}}\Big] dy\\
		&\geq \frac{\delta_w e^{-{x^2}/{8}} }{\sqrt{ 8 \pi}} \int_0^{x_w -\overline x_0}  
		e^{- {y^2}/{8}}   \Big[e^{{xy}/{4}} - e^{-{xy}/{4}}\Big] dy\\
		&\geq \frac{\delta_w e^{-{x^2}/{8} - {(x_w-\overline x_0)^2}/{8}}}{\sqrt{ 8 \pi}} 
		\frac{2}{x} \Big( \cosh\Big(\frac{x(x_w -\overline x_0)}{4}\Big) - 1 \Big)
		\geq \frac{x}{C}\delta_w e^ {- {x^2}/{8}},
\end{split}
\end{equation}
for some $C$ independent of all parameters, as long as $x_w \geq \overline x_0+1$.   We used here that $\cosh(x) - 1 \geq x^2 / C$ for some universal $C>0$.   On the other hand, from the explicit expression~\eqref{eq:sub-solution_v}
for~$v$, we get 
\[
v(2,x-\bar x_0) \leq C (x-\overline x_0) 
\exp\Big\{-x- x\farc{\gamma}{2}3^{\gamma-1}-
\farc{x^2}{6}+\farc{x\overline x_0}{3}\Big\}.
\]
Thus, there exists $\epsilon>0$ such 
that~
\[
\epsilon \tilde v(2,x)=\epsilon v(2,x-\bar x_0) \leq w(2,x)
\hbox{ for~$x \geq  \bar x_0= 3^\gamma+3$.}
\]

%
%

\paragraph{The subsolution.}

We now follow the same strategy as before, constructing a sub-solution
of the form $\underline  v(t,x) = \delta a(t)  \tilde v(t,x) $ on
\[
\mathcal{P} := \{(t,x) : t \geq 1, x > 2t + (1+t)^\gamma - 1\}.
\]
Another lengthy but straightforward computation shows that $\underline v(t,x)$
is a sub-solution for $w$ on $\mathcal P$ if we 
choose $a(t) = \exp\Big\{-\beta t^{2\gamma-1}\Big\}$ for a suitable $\beta$ and $\delta$ sufficiently small.


Note also that 
$\tilde v$ and $w$ satisfy the same boundary conditions  
at~$x = 2t + (1+t)^\gamma - 1$.  
Finally, choosing $\delta \leq \epsilon$ and using the computation 
\eqref{eq:w1_lower_bound} and the discussion following it, we see that
\[
\underline v(2,x) \leq w(2,x)\hbox{ for all $x > 3 + 3^\gamma$.}
\]
The conclusion of the proposition when $r \in (1,3)$ follows 
by simply using the 
explicit form of $\underline v(t,x)$.

\section{Estimates on the linearized KPP equation}\label{sec:estlinKPP}

In this section, we adopt the convention that any constant denoted $C$ may chance line-by-line but depends only on $\phi$ and $u_0$.

%
%
\subsection{The case $r>3$: the proof of \Cref{lem:self_similar}}\label{sec:self_similar}

The key observation is that $\gamma < 1/2$ when $r>3$.  Thus, the $t^\gamma$ term is of a lower order than
the diffusive scale $\sqrt{t}$.  
This allows us to use the strategy 
in~\cite{HNRR_Short}, obtaining energy estimates in self-similar variables.  
Since the present proof is similar to that in~\cite{HNRR_Short}, we provide a rather brief treatment.

\begin{proof}[Proof of \Cref{lem:self_similar}]
We begin by removing an exponential factor from $\tilde v$ and changing to the moving frame: let
\[
	z(t,x) := e^{x} \tilde v(t,2t + (t+t_0)^\gamma - t_0^\gamma + x),~~x>0.
\]
This function satisfies
\begin{equation}\label{eq:self_similar_z}
	\begin{aligned}
		&z_t = z_{xx} + \gamma(t+t_0)^{\gamma - 1} \Big(z_x - z\Big), \qquad &&\text{ for } t > 0, x > 0,\\
		&z(t,0) = 0,   &&\text{ for } t>0, \\
		&z(0,x) = e^{x} w_0(x), &&\text{ for } x > 0.
	\end{aligned}
\end{equation}
We now turn to self-similar variables, which are natural for the diffusive process. Let 
\begin{equation*}
\tau = \log\Big(1+\frac{t}{t_0}\Big), \qquad y = (t + t_0)^{-1/2}x,
\end{equation*}
and $\zeta(\tau, y) = z\Big(t_0(e^\tau - 1), t_0^{1/2}e^{\tau/2} y\Big)$. Then $\zeta$ satisfies the equation
\[
	\zeta_\tau =  \zeta_{yy} + \frac{y}2 \zeta_y +1  + \gamma(t_0e^\tau)^{\gamma-1/2} \zeta_y - 
	\Big( 1+ \gamma(t_0e^\tau)^{\gamma}\Big)\zeta. 
\]
We remove the integrating factor above, setting
\[
\zeta(\tau, y) = e^{- ( \tau + t_0^\gamma (e^{\gamma\tau}-1))} \bar\zeta(\tau, y),
\]
so that $\bar\zeta$ satisfies
\begin{equation}\label{oct402}
\bar\zeta_\tau = L\bar\zeta + \gamma t_0^{\gamma-1/2} e^{(\gamma-1/2)\tau}  \bar\zeta_y,
\end{equation}
with
\begin{equation}\label{oct408}
L:= \partial_y^2 + \frac{y}2\partial_y + 1.
\end{equation}
It is now heuristically clear that the last term in (\ref{oct402}) should be not important due to the $e^{(\gamma-1/2)\tau}$ term and the fact that $\gamma < 1/2$. The following lemma is proved in \Cref{sec:lemestpert}.
\begin{lemma}\label{lem:estperturbative}
Let $\bar\zeta$ solve
\[
	\bar\zeta_\tau = L\bar\zeta + \eps e^{(\gamma-1/2)\tau}  \bar\zeta_y, 
\]
with initial data $\bar\zeta(\tau=0,\cdot) = \bar\zeta_0$.  
There exists $\eps_0 > 0$ such that for all compact subsets $K \subset \R_+$ there exists $C_K >0$ such that for all $\eps < \eps_0$, 
\begin{equation*}
\bar\zeta(\tau,y) = y \Big( \frac{e^{-{y^2}/{4}}}{2\sqrt{\pi}} \Big( \int_0^\infty \xi \bar\zeta_0(\xi) d\xi + O(\eps) \Big) 
+ e^{(\gamma-1/2)\tau} \bar h(\tau,y) \Big),
\end{equation*}
for all $y>0$, $\tau>0$, and such that $\vert \bar h(\tau,y) \vert \leq C_K$ for all $\tau >0$ and $y \in K$.
\end{lemma}
Undoing the various changes of variable, we get
\begin{eqnarray}\label{oct406}
&&\tilde v(t,2t + (t+t_0)^\gamma - t_0^\gamma + x)
= e^{-x} z(t,x)
= e^{-x} \zeta\Big( \log\Big( 1+ \frac{t}{t_0}\Big), \frac{x}{(t+t_0)^{1/2}} \Big)\\
&&=  \frac{xe^{-x} t_0e^{- (( t+ t_0)^\gamma - t_0^\gamma)} }{(t+t_0)^{3/2}}
\Big( \frac{e^{-\frac{x^2}{4(t+t_0)}}}{2\sqrt{\pi}} \Big( \int_0^\infty \xi 
e^\xi w_0(\sqrt{t_0}\xi) d\xi + O\Big(t_0^{\gamma-\frac12}\Big) \Big) + \Big(1+ \frac{t}{t_0}\Big)^{\gamma-\frac12} h(t,x)\Big),
\nonumber
\end{eqnarray}
where $h(t,x) = \bar h\Big(\log\Big( 1+ \frac{t}{t_0}\Big),(t+t_0)^{-\frac12}x\Big) 
$.

First, notice that the $L^\infty$ bound on $\tilde v$ in \Cref{lem:self_similar} follows immediately from the expression above on sets of the form $[2t + t^\gamma, 2t + t^\gamma + \sigma \sqrt t]$.  To obtain bounds on sets of the form $[2t + t^\gamma + \sigma\sqrt t, \infty)$, 
we simply use that $e^{-t}\tilde v$ is a sub-solution to the heat equation on $\R$ (see, e.g.,~\eqref{eq:heat_subsolution}).  Hence, we obtain that, for $x\geq 0$,
\begin{equation}\label{eq:c102}
	\tilde v(t,2t + t^\gamma + \sigma\sqrt t+x)
		\leq \frac{Ce^t}{\sqrt t} \exp\Big\{-\frac{(2t + t^\gamma + \sigma\sqrt t+ x)^2}{4t}\Big\}
		\leq C e^{ - \sqrt t - t^\gamma },
\end{equation}
where $C$ is some constant depending only on the initial data and $\gamma$.

Second, we have
\begin{equation*}
\int_0^\infty \xi e^\xi w_0(t_0^{1/2}\xi) d\xi + O( t_0^{\gamma-\frac12})   
\geq \delta_w \int_0^{{x_w}/{t_0^{1/2}}} \xi e^\xi d\xi+ O( t_0^{\gamma-\frac12}).
\end{equation*}
Choosing first $x_w \geq \sqrt t_0$ and $t_0 \gg 1$ so that the first two terms in the parentheses 
in (\ref{oct406}) are positive and then choosing $T_0$ large depending on $t_0$ and $\alpha$, 
we have that, for all $0 \leq x \leq \sqrt{t + t_0}$ and $t \geq T_0$,
\[
	\tilde v(t,2t + (t+t_0)^\gamma - t_0^\gamma + x)
		\geq \frac{x}{C} \frac{e^{-x -(( t+ t_0)^\gamma - t_0^\gamma)} }{(t+t_0)^{3/2}}.
\]
The lower bound on $\tilde v(t,2t+t^\gamma + \sqrt{t})$ is immediate after evaluating at $x= t^\gamma - \Big( (t+t_0)^\gamma - t_0^\gamma \Big) + \sqrt{t}$.  This concludes the proof.

\end{proof}
\subsection{The case $r =3$: the proof of \Cref{lem:self_similar_critical}}\label{sec:subsolr=3}

Note that in this case $\gamma = 1/2$. As a consequence, the drift induced by the moving boundary has the same order as the diffusion.  
It is thus useful to modify the $t^\gamma$ term in the moving boundary by a small multiplicative factor.

\begin{proof}[Proof of \Cref{lem:self_similar_critical}]

To begin, fix $\epsilon>0$. Work in the moving frame $2t +  [ ( t + 1)^{1/2} - 1 ]$ and remove an exponential factor, as previously: \begin{equation*}
z(t,x) := e^x \tilde v(t, x + 2t + [ (1+t)^{1/2} - 1 ]).
\end{equation*}
Passing then to self-similar coordinates
\begin{equation*}
\tau = \log(t+1) \quad \text{ and } \quad  y = (t + 1)^{-1/2}x,
\end{equation*}
so that 
\[
	\zeta(\tau, y) := z\Big( e^\tau - 1, e^{{\tau}/{2}} y\Big),
\]
we see that $\zeta$ satisfies
\[
	\zeta_\tau = L\zeta + \frac{1}{2}\zeta_y - \Big(1 + \frac{1}{2}e^{{\tau}/{2}}\Big) \zeta,
\]
with $L$ as in (\ref{oct408}). Finally, pulling out the zeroth order factor 
\[
	\zeta(\tau,y) = e^{- \tau - (e^{{\tau}/{2}} - 1)} \bar \zeta(\tau,y),
\]
we see that $\bar \zeta$ solves
\begin{equation}\label{eq:c001}
	\bar\zeta_\tau = L\zeta + \frac{1}{2}\zeta_y.
\end{equation}

We finish using the following lemma, proved in \Cref{sec:lemestpert}.  This result falls outside 
of~\cite{HNRR_Short} and \Cref{lem:estperturbative} because the $\bar\zeta_y$ term in~\eqref{eq:c001} is no longer a remainder term.
\begin{lemma}\label{lem:estperturbativeeps}
Let $\bar\zeta$ solve (\ref{eq:c001}), then it can be represented as
\begin{equation}\label{eq:c002}
	\bar\zeta(\tau,y)
		= \exp\Big\lbrace - \frac{y^2}{8} - \frac{y}{4} \Big\rbrace
			\Big(\Big(\int_{\R^+} \psi(y) e^{\frac{y^2}{8} + \frac{y}{4}}  \bar\zeta(0,y) \, dy \Big) \psi(y) e^{-\lambda \tau} + y\bar h(\tau,y) e^{-\mu \tau}\Big).
\end{equation}
Here, $\mu > \lambda > 0$,
$\|\psi\|_2 = 1$, and $\psi$ and $\overline h(\tau,y)$ are bounded on all compact subsets of $[0,\infty)$.  Also, for every compact set $K\subset[0,\infty)$, there exists $C_K>0$ such that $C_K^{-1} y \leq \psi(y) \leq C_K y$. for all $y\in K$.
\end{lemma}

We are now able to establish the upper bound on $\tilde v$ for all $x \in [2t + ((1+t)^{1/2} - 1), 2t + 2 \sqrt t]$.  Indeed, returning to the original variables in~\eqref{eq:c002}, we find
\begin{equation}\label{eq:c3}
\begin{split}
	|\tilde v(t,x)|
		&= e^{-(1+t)^{1/2} y} \zeta(\tau, y)
		= e^{-(1+t)^{1/2} y - \tau - (e^{\tau/2} - 1)} \bar \zeta(\tau, y)\\
		&\leq C e^{-(1+t)^{1/2} y - \tau - (e^{\tau/2} - 1) - \frac{y^2}{8} - \frac{y}{4}} \left( \psi(y) e^{-\lambda \tau} + h(\tau,y) e^{- \mu \tau}\right)
		\leq Ce^{- \sqrt t}.
\end{split}
\end{equation}

In fact, the estimate~\eqref{eq:c3} holds for all $x\geq 2t + [(1+t)^{1/2} - 1]$ since, as above, $\tilde v$ may be estimated  for $x\in [2t + 2 \sqrt t, \infty)$ using the same approach as in~\eqref{eq:c102} (see also~\eqref{eq:heat_subsolution}).  Indeed,
\[
	\tilde v(t,x)
		\leq \frac{C \sqrt t}{x} e^{t - \frac{x^2}{2t}}
		\leq \frac{C}{\sqrt t} e^{t - \frac{(2t + 2 \sqrt t)^2}{4t}}
		= \frac{C}{\sqrt t} e^{- 2\sqrt t - 1}
		\leq Ce^{- \sqrt t}.
\]

We now establish the lower bound from~\eqref{eq:c002}.  Taking $t$ sufficiently large and evaluating at $x = 2t + 2 \sqrt t$, we see that
\[
	\tilde v(t, 2t + 2\sqrt t)
		\geq \frac{\alpha}{t^{1 + \lambda}} e^{-2\sqrt t},
\]
for some $\alpha$ depending only on $u_0$.  This concludes the proof.

\end{proof}

\subsection{The case $r\in (1,3)$: the proof of \Cref{lem:sub-solution}}\label{sec:sub-solution}

To motivate some of the steps in the following proof, we briefly discuss a heuristic.  In the stationary frame, as we did in~\eqref{eq:heat_subsolution}, 
we may always estimate $\tilde v$ above by ignoring the Dirichlet boundary condition and using the fact that $e^{-t}\tilde v$ solves the heat equation.  Thus,
\[
 \tilde v(t,x + 2t + t^\gamma)
		\lesssim t^{-1/2} \exp\Big\{t - \frac{(x + 2t + t^\gamma)^2}{4t}\Big\}
		= t^{-1/2} \exp\Big\{- x - \frac{x^2}{4t} - \frac{x}{\sqrt{t}}\frac{t^{\gamma - 1/2}}{2} 
		- t^\gamma - \frac{t^{2(\gamma-1/2)}}{4} \Big\}.
\]
Recalling that $\gamma > 1/2$, we see that on the diffusive scale $x \sim \sqrt t$, the Gaussian term ${x^2}/{4t}$ 
and the~$t^{-1/2}$ in front are (much) lower order and, thus, negligible, but all other terms are large. Hence, our sub-solution should contain all 
such terms to be reasonably sharp.  In particular, while the $xt^{\gamma-1}$ term appears small at first glance since $\gamma < 1$,  it is not 
negligible in the diffusive scale $x \sim \sqrt t$.  While the terms depending only on $t$ show up as obvious integrating factors, this term will 
not.  Hence, the key to the proof below is in carefully taking account of this term.  Note that here we see the effect of  $\gamma > 1/2$.  

\begin{proof}[Proof of \Cref{lem:sub-solution}]

We show how to ``guess'' the form of the sub-solution $v$. 
We begin by removing an exponential from $\tilde v$ and changing to the moving frame. Define, for $x\in \R^+$, 
\[
	z(t,x) := e^{x} \tilde v(t,2t + (t+1)^\gamma - 1 + x),
\]
so that~\eqref{eq:linearized} becomes
\begin{equation}\label{eq:self_similar_v}
	\begin{aligned}
&		z_t \leq z_{xx} + \gamma(t+1)^{\gamma - 1}(z_x - z), ~~ t > 0, x > 0,\\
&		z(t,0) = 0,   \\
&		z(0,x) = e^{x} w_0(x).
	\end{aligned}
\end{equation}
Turning to self-similar variables,  
\begin{equation*}
	\tau = \log(1+t),
		\quad y = (t + 1)^{-1/2}x,
		\quad \text{ and } \quad \zeta(\tau, y) = z ( e^\tau - 1, e^{\tau/2} y ),
\end{equation*}
we wish to construct $\zeta$ that satisfies the inequality
\begin{equation}\label{oct412}
	\zeta_\tau \leq \zeta_{yy} + \frac{y}2 \zeta_y + \gamma e^{(\gamma-1/2)\tau}\zeta_y -   \gamma e^{\gamma\tau} \zeta. 
\end{equation}
%
As $\gamma > 1/2$, the drift in (\ref{oct412}) is not a perturbation anymore. The heuristic discussion preceding this proof  
indicates that we should consider 
\[
\zeta(\tau, y) = e^{-\alpha y e^{(\gamma-1/2)\tau}} \psi(\tau,y),
\]
with $\alpha\in \R$ to be determined.  Then we require
\begin{equation}\label{eq:chris20}
	\psi_\tau
		 \leq L\psi + (\gamma - 2\alpha) e^{(\gamma-1/2)\tau} \psi_y 
		 - \Big(1 + \alpha(\gamma - \alpha) e^{(2\gamma-1)\tau}   + \gamma e^{\tau\gamma}\Big) \psi - \alpha (1-\gamma) y 
		 e^{(\gamma - {1}/{2})\tau} \psi.
\end{equation}
with $L$ as in (\ref{oct408}). 
%
To remove the drift term, we set $\alpha = {\gamma}/{2}$.  Then~\eqref{eq:chris20} becomes
\[
	\psi_\tau
		-L\psi
		+ \Big(1 + \frac{\gamma^2}{4} e^{(2\gamma-1)\tau} 
		+ \gamma e^{\tau\gamma}\Big) \psi
		+ \frac{\gamma}{2} (1-\gamma) y e^{(\gamma - {1}/{2})\tau} \psi
		\leq 0.
\]
Further, writing 
\[
\psi(\tau, y) = \exp\Big\{ -\tau - e^{\gamma \tau} - \frac{\gamma^2}{4(2\gamma - 1)} e^{(2\gamma - 1)\tau}\Big\} \Psi(\tau,y),
\]
we arrive at
\begin{equation}\label{eq:chris21}
\Psi_\tau - L\Psi + \frac12 \gamma (1-\gamma) y e^{(\gamma - {1}/{2})\tau} \Psi \leq 0.
\end{equation}
To deal with the last term in~\eqref{eq:chris21}, 
let $a,a',b >0$ be constants to be determined and define
\[
	\Psi(\tau,y) = y\exp\Big\{-a e^{\tau(2\gamma-1)} - a'\tau - \farc{y^2}{b}\Big\}.
\]
By a direct computation, we see that
\begin{equation}\label{oct416}
\begin{split}
	\Psi_\tau - &L\Psi + \frac{\gamma}{2}(1-\gamma) y e^{(\gamma - {1}/{2})\tau} \Psi\\
		&=\Big[ -a' - a (2\gamma -1)e^{\tau(2\gamma-1)}
			- \frac{y^2}{b} \Big(\frac{4}{b} - 1\Big)
			+ \Big(\frac{6}{b} - \frac{3}{2}\Big)
			+ \frac{\gamma(1-\gamma)}{2} y e^{(\gamma - {1}/{2})\tau} \Big] \Psi.
\end{split}
\end{equation}
It is clear that to have (\ref{eq:chris21}), 
we must choose $b<4$.  For simplicity, we take $b = 2$, and
\[
a' = \frac{6}{b} - \frac{3}{2} = \frac32,
\]
so that (\ref{oct416}) becomes
\[ 
	\Psi_\tau - L\Psi + \frac{\gamma}{2}(1-\gamma) y e^{(\gamma - {1}/{2})\tau} \Psi
		=\Big[ - a (2\gamma -1)e^{\tau(2\gamma-1)}
			+ \frac{\gamma(1-\gamma)}{2} y e^{(\gamma - {1}/{2})\tau}
			- \frac{y^2}{2}
 \Big] \Psi.
\]
The choice 
\[
a \geq \frac{\gamma^2(1-\gamma)^2}{8(2\gamma-1)}
\]
ensures that  (\ref{eq:chris21}) holds.
Returning to our original variables, we see that
\begin{align*}
&v(t,x)
	= \tilde v(t,2t + (t+1)^\gamma - 1 + x)
	= e^{-x} \zeta(\log(1+t),(t + 1)^{-1/2}x)\\
& = e^{-x}e^{-\alpha (t + 1)^{-1/2}x (1+t)^{ \gamma-1/2 }} \psi(\log(1+t),(t + 1)^{-1/2}x)\\
& = \frac{1}{1+t}\exp\Big\{-x - \frac{\gamma}{2}x \Big(1+t\Big)^{ \gamma-1} - 
(1+t)^\gamma - \frac{\gamma^2(1+t)^{2\gamma - 1}}{4(2\gamma - 1)} \Big\} \Psi(\log(1+t),(t +1)^{-1/2}x)\\
&= \frac{x}{(1+t)^3}\exp\Big\{-x - \frac{\gamma}{2} x (1+t)^{ \gamma-1} - 
(1+t)^\gamma - \Big[\frac{\gamma^2}{4(2\gamma - 1)} + a\Big](1+t)^{2\gamma - 1}  - \frac{x^2}{2(1+t)} \Big\}.
\end{align*}
This concludes the proof.

\end{proof}

\section{The Fisher-KPP equation with a Gompertz non-linearity}\label{sec:local_equation}

A side effect of our analysis gives the asymptotics for a related local equation:
\begin{equation}\label{eq:local}
	u_t - \Delta u = f_r(u).
\end{equation}
Here, we assume that $f_r \in C^1$, $r \in (1,\infty)$,
and there exist positive constants $\theta_f$, $\delta_f$, and $A_f$ such that
\begin{equation}
f_r(0) = 0, \quad
          f_r(u) >0 \text{ for all } u \in (0,\theta_f),
          \quad f_r(\theta_f) = 0,
        \quad f_r(u) = 0 \text{ for all } u \geq \theta_f,
\end{equation}
and
\begin{equation}\label{eq:local_f}      
u \Big( 1 - A_f \log\Big(\frac{1}{u}\Big)^{1-r} \Big)  \leq f_r(u)  \leq u \Big( 1 - A_f^{-1} \log\Big(\frac{1}{u}\Big)^{1-r} \Big), 
\end{equation}
for $u \in (0,\delta_f)$.
\begin{theorem}\label{thm:main_local_delay}
    Suppose that  the initial condition $u_0(x)$ for~\eqref{eq:local} 
    is as in~\eqref{eq:u_0}.
        If $r>3$, then the solution $u(t,x)$ propagates with a logarithmic delay:
    \begin{equation}\label{eq:local_log_delay_above}
        \lim_{L\to\infty} \limsup_{t\to\infty} \sup_{x \geq L} \, u\Big(t,2t - \frac{3}{2}\log t + x\Big) = 0,
    \end{equation}
    and
    \begin{equation}\label{eq:local_log_delay_below}
       \lim_{L\to\infty}\limsup_{t\to\infty}\sup_{x \leq -L} \Big|u\Big(t,2t - \frac{3}{2} \log t + x\Big) - \theta_f\Big| = 0.
    \end{equation}
If $ r= 3$,  $u(t,x)$ propagates with a larger logarithmic delay: there exists $S_A>s_A>3/2$ such that 
	\begin{equation}\label{eq:big_local_log_delay_below}
		\liminf_{t\to\infty} \sup_{x \leq 0 } \Big|u\Big(t,2t - S_A \log t + x\Big) - \theta_f\Big| = 0,
	\end{equation}
	and
	\begin{equation}\label{eq:big_local_log_delay_above}
		\lim_{t\to\infty} \sup_{x \leq 0 } u\Big(t,2t - s_A \log t + x\Big) = 0,
	\end{equation}
If $r \in (1,3)$, then the delay is algebraic: there exist $C_f > c_f >0$, depending only on $f_r$, such that
    \begin{equation}\label{eq:local_alg_delay_above}
        \lim_{t\to\infty} \sup_{x \geq 0} u\Big(t,2t - c_f t^\frac{3-r}{1+r} +x\Big) = 0,
    \end{equation}
    and
    \begin{equation}\label{eq:local_alg_delay_below}
        \lim_{t\to\infty}\sup_{x \leq 0} \Big|u\Big(t,2t - C_f t^\frac{3-r}{1+r} + x\Big) - \theta_f\Big| = 0.
    \end{equation}      
\end{theorem}

The proof of \eqref{eq:local_log_delay_above} follows directly from \Cref{sec:lower_bound}. The proofs of \eqref{eq:local_log_delay_below},~\eqref{eq:big_local_log_delay_below}, and \eqref{eq:local_alg_delay_below} follow from what was done in \Cref{sec:estw}, combined with a standard argument saying that the convergence is necessarily to the steady state $\theta_f$ (see, e.g., \cite{HNRR_Short}).  The bounds~\eqref{eq:big_local_log_delay_above} and \eqref{eq:local_alg_delay_above} need
additional ingredients; the work is similar so we only detail the computations for~\eqref{eq:local_alg_delay_above}. Indeed, since our non-linearity is local, we cannot ``pull'' information from the front as we did above when we used the value of $u$ at the front to bound $\phi \star u$ far ahead of the front.  In order to get around this, we state a weak lower bound on $u$.  

\begin{lemma}\label{lem:local_alg_lower}
Let the hypotheses of \Cref{thm:main_local_delay} be satisfied. Then there exists $\delta_f>0$, depending only on $f$, such that
    \[
        u(t,x) \geq \exp \{ - \delta_f t^{ \gamma }\}
    \]
for all $t$ sufficiently large and all $x \leq 2t + t^\gamma$, where we again define $\gamma = 2/(1+r)$.
\end{lemma}

Such a bound follows from the analysis of the lower bound in part (3) of \Cref{prop:g} and requires no new ideas.  As such, we omit the proof. 

The main point in the proof of \Cref{thm:main_local_delay} is to use the lower bound in \Cref{lem:local_alg_lower} on $u$ along with 
the form of the non-linearity to replace the estimate of $\phi \star u$ that we used in the proof of the upper bound in \Cref{thm:main_delay} when $r \in (1,3)$.


\begin{proof}[Proof of~\eqref{eq:local_alg_delay_above} assuming \Cref{lem:local_alg_lower}]

We use a super-solution
\[
\overline v(t,x) := B \exp \Big\{ - \Big( x - 2t + 2c_f t^{2\gamma-1}\Big)\Big\},
\]
with $c_f>0$ to be determined.  Then $\overline v$ satisfies
\[
    \overline v_t = \overline v_{xx} + \overline v\Big(1 - 2c_f (2\gamma-1) t^{2\gamma-2}\Big).
\]
On the other hand, using the bound on $f$~\eqref{eq:local_f} along with \Cref{lem:local_alg_lower}, we have that, for all $t$ sufficiently large and $x \leq 2t + t^\gamma$, 
\[\begin{split}
    u_t - u_{xx} = f_r(u)
        &\leq u \Big(1 - A_f \log\Big(\frac{1}{u(t,x)}\Big)^{1-r}\Big)
        \leq u\Big(1 - A_f \delta_f^{1-r} t^{\gamma(1-r)}\Big).
\end{split}\]
Recalling $2\gamma-2= \gamma(1-r)$, and choosing $c_f$ such that $A_f \delta_f^{r-1} \geq 2 c_f (2\gamma-1)$, we see that $\overline v$ is 
a super-solution for $u$.  
\end{proof}

\section{The local-in-time Harnack inequality: \Cref{lem:Harnack}}\label{sec:harnack}



\begin{proof}[Proof of \Cref{lem:Harnack}]
Up to a shift in time, we may assume that $t = 0$.
We may also assume that $c\equiv 0$.  Indeed, let 
\[
u_{\pm}(t,x) = e^{\pm t\|c\|_{L^\infty([0,T]\times\R)}} w (t,x),
\]
where $w$ solves the heat equation
\[ 
w_t = w_{xx}, 
\]
with the initial condition $w(t=0,x) = u(t=0,x)$.
Then $u_+$ is a super-solution to $u$ while $u_-$ is a sub-solution to $u$.  
Hence, we have 
\[
    \frac{u(T,x+y)}{\|u_-\|_{L^\infty}^{1-1/p} u(T,x)^{1/p}}
        \leq \frac{u_+(T,x+y)}{\|u_-\|_{L^\infty}^{1-1/p} u_-(T,x)^{1/p}}
        \leq e^{2 \|c\|_{L^\infty} T} \frac{w(T,x+y)}{\|w\|_{L^\infty}^{1-1/p} w(T,x)^{1/p}}.
\]
In view of this inequality, it is enough to prove the claim for $w$, that is, solutions to the heat equation.

Let $G$ be the one-dimensional heat kernel $G(t,x) = (4\pi t)^{-1/2}e^{-x^2/(4t)}$.  Fix $s = (p+1)/2p$, notice that $s \in (0,1)$ and $sp > 1$, and let $q$ be the dual exponent of $p$.  Then we have  
\begin{equation}\label{eq:Harnack1}
\begin{split}
    w(T,x+y)
        &= \int_\R w(0,z) G(T,x+y-z) dz\\
        &\leq \|w\|_\infty^{1-1/p} \int_\R 
        w(0,z)^{1/p} G(T,x+y-z)^s G(T,x+y-z)^{1-s} dz\\
        &\leq \|w\|_{\infty}^{1-1/p} \Big(\int_\R w(0,z) G(T,x+y-z)^{sp}dz \Big)^{1/p} \Big\|G^{1-s}(T,\cdot)\Big\|_q\\
        &\leq C_p T^{\frac{s}{2} - \frac{1}{2p}}
         \|w\|_{\infty}^{1-1/p} \Big(\int_\R w(0,z) G(T,x+y-z)^{sp}dz \Big)^{1/p}.
\end{split}
\end{equation}
Above, we have used that
\begin{align*}
\Big\|G^{1-s}(T,\cdot)\Big\|_q &= (4\pi T)^{-\frac12 (1-s)} \left( \int_{\R^n} e^{-q(1-s)\frac{x^2}{4T}}\, dx \right)^{\frac{1}{q}}\\&= (4\pi T)^{-\frac12 (1-s)} \left( \left( \frac{4T\pi}{q(1-s)} \right)^\frac12 \right)^{\frac{1}{q}} = C_p T^{ \frac{1}{2q} - \frac12 (1-s)}= C_p T^{ \frac{s}{2} - \frac{1}{2p}}.
\end{align*}
We now seek a bound on $G(T,x+y-z)^{sp}$ in terms of $G(T,x-z)$.  
To this end, we recall that~$sp > 1$, let $x' = x-z$ and we compute
\[\begin{split}
    \frac{G(T,x' + y)^{sp}}{G(T,x')}
        &= (4\pi T)^{\frac{(1 - sp)}{2}} 
        \exp \Big\{ - \frac{sp (x'+y)^2}{4T} + \frac{|x'|^2}{4T} \Big\}\\
        &= (4\pi T)^{\frac{(1 - sp)}{2}} \exp \Big\{ - \frac{sp |x'|^2}{4T} - \frac{sp x'y}{2T} - \frac{spy^2}{4T} + \frac{|x'|^2}{4T} \Big\}\\
        &= (4\pi T)^{\frac{(1 - sp)}{2}} \exp \Big\{ - \frac{(sp -1) |x'|^2}{4T} - \frac{sp x'y}{2T} - \frac{spy^2}{4T}\Big\}\\
        &\leq (4\pi T)^{\frac{(1 - sp)}{2}} \exp \Big\{ - \frac{(sp -1) |x'|^2}{4T} + \frac{(sp-1)|x'|^2}{4T} + \frac{(sp)^2 y^2}{4T(sp-1)} - \frac{spy^2}{4T}\Big\}\\
        &= (4\pi T)^{\frac{(1 - sp)}{2}} \exp \Big\{ \frac{(sp)^2 y^2}{4T(sp-1)} - \frac{spy^2}{4T}\Big\}
        = (4\pi T)^{\frac{(1 - sp)}{2}} \exp \Big\{ \frac{spy^2}{4T (sp-1)}\Big\}.
\end{split}\]
Define $\beta = \frac{sp}{4(sp-1)}$.  Using the above bound in~\eqref{eq:Harnack1}, we obtain
\[\begin{split}
    w(t,x+y)
        &\leq C_p e^{\frac{\beta y^2}{T}} T^{ \frac{s}{2} - \frac{1}{2p} + 
        \frac{(1-sp)}{2p}} \|w\|_{\infty}^{1-1/p} \Big(\int_\R w(0,z) G(T,x-z)dz \Big)^{1/p}\\
        &=C e^{{\beta y^2}/{T}} \|w\|_{\infty}^{1-1/p} w(t,x)^{1/p}.
\end{split}\]
In the second line we used the explicit choice of $s$ to simplify the exponent of $T$.  This concludes the proof.
\end{proof}

\appendix

\numberwithin{equation}{section}

\section{Proofs of \Cref{lem:r3_supersolution,lem:estperturbative,lem:estperturbativeeps}.}
\label{sec:lemestpert}

\begin{proof}[{Proof of \Cref{lem:estperturbative}}]
The proof of this lemma is similar to that of a corresponding estimate 
in~\cite{HNRR_Short}. 
However, the proof there only deals with moving boundary conditions of the form $2t + r\log(t)$.  Hence, for completeness, we provide a streamlined proof.  Recall that $\bar\zeta$ solves
\[
\bar\zeta_\tau = L\bar\zeta + \eps e^{(\gamma-1/2)\tau}  \bar\zeta_y. 
\]
To rectify the fact that the operator $L$ is not self-adjoint, 
we remove a Gaussian term. Let
\[
\bar\zeta(\tau, y) = e^{-y^2/8} \zeta^*(\tau, y),
\]
then $\zeta^*$ satisfies
\begin{equation}\label{eq:self_similar_equation}
	 \zeta_\tau^* + M \zeta^* =  \eps e^{(\gamma-1/2)\tau} 
	 \Big(\bar\zeta_y^*  - \frac{y}{4}\bar\zeta^* \Big),
\end{equation}
where 
\[
M\zeta^*:= -\zeta_{yy}^* + \Big(\frac{y^2}{16} - \frac34 \Big)\zeta^*.
\]
The principle eigenvalue of $M$ 
is associated to the eigenfunction 
\[
\psi(y) := (2\sqrt\pi)^{-1/2} y e^{-y^2/8}.
\]
%
Define the non-negative quadratic form 
\begin{equation*}
Q(f) := \langle Mf, f\rangle =  \int_\R \Big( f_y^2 + \Big( \frac{y^2}{16} - \frac{3}{4}\Big)f^2 \Big)dy,
\end{equation*}
for all $f \in H^1(0,\infty)$ such that $yf \in L^2(\R^+)$. 

Multiplying~\eqref{eq:self_similar_equation} by $\zeta^*$ and integrating, we obtain
\[\begin{split}
	\partial_\tau \,\Vert \zeta^* \Vert_{L^2(\R^+)}^2 + 2Q(\zeta^*)
		&= - 2\eps e^{(\gamma-1/2)\tau} 
		\int_0^\infty \frac{y}{4}(\zeta^*)^2 dy \leq 0.
\end{split}\]
Hence $\zeta^*$ is bounded uniformly in $L^2$ independently of $\tau$.
%
Next, let $\zeta_1^* = \langle \psi, \zeta^*\rangle$.  We have  
\begin{equation}\label{eq:bar_zeta_1}
|\partial_\tau \zeta_1^*|
		\leq \eps e^{-(\gamma-1/2)\tau}
		( \vert \langle (-\psi_y), \zeta^*\rangle\vert + \vert  \zeta_1^*\vert ) 
		\lesssim \eps e^{-(\gamma-1/2)\tau} \|\zeta^*\|_2.
\end{equation}
Integrating this inequality in $\tau$ and using the $L^2$ bound above, we obtain
\begin{equation}\label{eq:chris26}
	\vert \zeta_1^*(\tau) - \zeta_1^*(0) \vert
		\lesssim \eps \|\bar\zeta(0,\cdot)\|_2.
\end{equation}
We now show that the component of $\zeta^*$ that is orthogonal to $\psi$ decays in time.  Let 
\[
\zeta^{*\perp} := \zeta^* - \zeta_1^* \psi,
\]
then 
\[
Q(\zeta^{*\perp}) \geq \frac12 \|\zeta^{*\perp}\|_2^2.
\]
 Using~\eqref{eq:self_similar_equation}, we   obtain  
\begin{equation*}
	\partial_\tau \Vert \zeta^{*\perp} \Vert_{L^2(\R^+)}^2 + 2Q(\zeta^{*\perp}) 
		\lesssim \eps e^{-(\gamma-1/2)\tau}\|\zeta^{*}(0,\cdot)\|_2\|\zeta^{*\perp}\|_2,
\end{equation*}
%
from which we deduce that 
\[
	\|\bar\zeta^\perp(\tau,\cdot)\|_2 \lesssim e^{-(\gamma-1/2) \tau} \|\bar\zeta(0,\cdot)\|_2.
\]
Gathering all estimates concludes the proof.
%
\end{proof}

\begin{proof}[{Proof of \Cref{lem:estperturbativeeps}}]

Recall that $\bar\zeta$ solves
\[
	\bar\zeta_\tau = L\bar\zeta + \frac{1}{2} \bar\zeta_y. 
\]
To pass to a self-adjoint form, write
\[
	\bar\zeta(\tau,y) = \exp\Big\lbrace - \frac{y^2}{8} - \frac{ y}{4} \Big\rbrace  \zeta^*(\tau,y),
\]
so that $\zeta^*$ solves
\[
	\zeta_\tau^* + M_1 \zeta^* = 0,
\]
where, 
\[
M_1 \zeta^* := -\zeta_{yy}^* + \Big[ \Big(\frac{y^2}{16} - \frac34 \Big) +  \Big(\frac{y}{8} + \frac{1}{16}\Big) \Big] \zeta^* = M\zeta^* + \Big(\frac{y}{8} + \frac{1}{16}\Big) \zeta^*.
\]
This operator is now self-adjoint with a compact resolvent. 
Let $\psi$ and $\lambda$ be the principal eigenfunction and eigenvalue of the operator above satisfying the boundary condition $\psi(0) = 0$ and the normalisation $\Vert \psi \Vert_{L^2(\R^+)} = 1$.  

Observe that
\begin{equation*}
Q(f) := \langle M f, f\rangle = \langle Mf, f\rangle + \Big\langle \Big(\frac{y}{8} + \frac{1}{16}\Big)f, f \Big\rangle \geq 0,
\end{equation*}
and thus $\lambda > 0$. 

Write 
\[
\zeta^* := \langle \psi, \zeta^*\rangle\psi + \zeta^{*\perp},
\]
so that
\begin{equation}\label{oct502}
Q(\zeta^{*\perp}) \geq \mu \|\zeta^{*\perp}\|_2^2,
\end{equation}
where $\mu$ is the second eigenvalue of $M$.
After a time differentiation we have   
\[
\langle \psi_\eps, \zeta^*\rangle(\tau) =\langle \psi, \zeta^*\rangle(0) e^{-\lambda \tau},
\]
and
%
%
%
%
%
as a consequence of (\ref{oct502}), 
that 
\[
\|\zeta^{*\perp}\|_2 (\tau) \leq \|\zeta^{*\perp}\|_2 (0) e^{-\mu \tau}.
\]
Then, locally we have $\|\zeta^{*\perp}\|_\infty (\tau) \lesssim e^{-\mu \tau}$ by parabolic regularity.
%
This yields 
\begin{equation}\label{eq:c2}
	\bar\zeta
		:= \exp\Big\lbrace - \frac{y^2}{8} - \frac{y}{4} \Big\rbrace \Big( \Big(\int_{\R^+} \psi(y) \exp\Big\lbrace  \frac{y^2}{8} + \frac{y}{4} \Big\rbrace \bar\zeta(0,y) \, dy \Big) \psi(y)  e^{-\lambda \tau} + \bar h(\tau,y) e^{-\mu \tau}\Big),
\end{equation}
where $\bar h$ is bounded in $\tau$, locally in $y$.  To finish, we simply note that, by elliptic regularity theory, for any compact set $K\subset [0,\infty)$, there exists $C_K>0$ such that
\[
	\frac{y}{C_K^{-1}}
		\leq \psi(y)
		\leq C_K y
		\qquad\text{ for all } y \in K
\]
since $\psi(0) = 0$.  This concludes the proof.
%

\end{proof}

\begin{proof}[Proof of \Cref{lem:r3_supersolution}]

As before, we pass to the moving frame and remove an exponential: let
\[
	v(t,x) = e^{-x +(2t - \tilde s_\phi\log(t+t_0))}\overline v(t,x-(2t - \tilde s_\phi\log(t+t_0))).
\]
Then, we want
\[
	\bar v_t
		\geq \bar v_{xx} - \bar v_x \frac{\tilde s_\phi}{t + t_0} + \bar v\left( \frac{\tilde s_\phi}{t+t_0} - \nu\right).
\]
Changing now to self-similar variables and recalling the definition $\e = t_0^{-1/2}$, let
\[\begin{split}
	&\tau = \log\left( 1 + \frac{t}{t_0}\right),
		\quad y = \frac{x}{\sqrt{t+t_0}},
		\quad \zeta(\tau, y) = \bar v\left( t_0 (e^\tau - 1), t_0^{1/2} e^{\tau/2} y\right),\\
		&\text{ and } \quad
		N(\tau,y) := t_0 e^\tau \bar \nu\left( t_0 (e^\tau - 1), t_0^{1/2} e^{\tau/2} y\right)
			= \frac{1}{CA_\phi \left(y + \e e^{-\tau/2}((S_\phi - \tilde s_\phi) (\tau + \log(t_0)) + L) \right)^2}.
\end{split}\]
The important point here is that $N(\infty,\cdot)$ is neither infinity nor zero as it would be for another choice of $r$.  This is what induces the larger delay.

Then we must find $\zeta$ that satisfies:
\[
	\zeta_\tau
		- L\zeta + \left(1 - \tilde s_\phi + N(0,y)\right)\zeta
		\geq \tilde s_\phi \e e^{-\tau/2} \zeta_y
			+ \zeta\left(N(0,y) - N(\tau,y)\right).
\]
We have changed the order to emphasize that the right hand side is a small error.

Let $\bar \zeta = e^{-\tau/2} e^{y^2/8} \zeta$.  The above yields
\begin{equation}\label{eq:edge_bar_zeta}
	\bar \zeta_\tau
		+ M \bar \zeta + \left(\frac{3}{2} - \tilde s_\phi + N(0,y)\right)\zeta
		\geq \tilde s_\phi \e e^{-\tau/2} \bar\zeta_y
			-  \frac{y}{4} \tilde s_\phi \e e^{-\tau/2} \bar \zeta
			+ \bar\zeta\left(N(0,y) - N(\tau,y)\right).
\end{equation}
We now define $\bar \zeta$.  As above $M + N(0,\cdot)$ is self-adjoint with a compact resolvent.  Let $\bar \zeta_0$ and $\lambda_0$ be its principle eigenelements.  Using the Rayleigh quotient and testing with $ye^{-y^2/8}$, we see immediately that $\lambda_0 > 0$.  We need only verify that $\bar \zeta$ satisfies~\eqref{eq:edge_bar_zeta}; indeed, setting $\tilde s_\phi = 3/2 + \lambda_0/2$,
\[\begin{split}
	\bar \zeta_\tau
		&+ M \bar \zeta + \left(\frac{3}{2} - \tilde s_\phi + N(0,y)\right)\zeta
			- \tilde s_\phi \e e^{-\tau/2} \bar\zeta_y
			+ \frac{y}{4} \tilde s_\phi \e e^{-\tau/2} \bar \zeta
			- \bar\zeta\left(N(0,y) - N(\tau,y)\right)\\
		&= \frac{\lambda_0}{2}\bar \zeta
			- \tilde s_\phi \e e^{-\tau/2} \bar\zeta_y
			+ \frac{y}{4} \tilde s_\phi \e e^{-\tau/2} \bar \zeta
			- \bar\zeta\left(N(0,y) - N(\tau,y)\right).
\end{split}\]
The first, third, and fourth terms are all positive.  A simple maximum principle argument yields that $- \bar \zeta_y$ is positive for all $y \geq y_0$ for some $y_0>0$.  The Hopf maximum principle implies $(\lambda_0/2)\bar\zeta - \tilde s_\phi \e e^{-\tau/2} \bar \zeta_y > 0$ for all $y > e^{-\tau/2}/C$ for some $C>0$.  Thus, $\bar \zeta$ satisfies~\eqref{eq:edge_bar_zeta} on $[e^{-\tau/2}/C,\infty)$, which, after translating back to physical variables, concludes the proof. 

\end{proof}

  \bibliographystyle{abbrv}
  \bibliography{refs}
\end{document}